\input ./style/arxiv-general.cfg
\documentclass[aop,MSNbibl,dvips]{arximspdf}
\makeatletter
   \@ifpackageloaded{graphicx}{}{\usepackage{graphicx}}
\makeatother

\doi{10.1214/14-AOP909} 
\volume{43}
\issue{3}
\pubyear{2015}
\firstpage{1535}
\lastpage{1576}

\makeatletter
\renewcommand{\shortmid}{\prime}
\newcommand{\rrVert}{\Vert}
\newcommand{\rrvert}{\vert}
\newcommand{\llVert}{\Vert}
\newcommand{\llvert}{\vert}
\newcommand{\eqref}[1]{(\ref{#1})}
\newtheorem{theorem}{Theorem}
\newtheorem{corollary}[theorem]{Corollary}
\newproclaim{definition}[theorem]{Definition}
\newtheorem{lemma}[theorem]{Lemma}
\newtheorem{proposition}[theorem]{Proposition}
\newproclaim{remark}[theorem]{Remark}
\makeatother

\begin{document}
\begin{frontmatter}

\title{Sobolev differentiable stochastic flows for SDEs with singular
coefficients: Applications to the~transport equation}
\runtitle{Sobolev flows for SDEs with singular coefficients}

\begin{aug}
\author[A]{\fnms{Salah-Eldin A.} \snm{Mohammed}\corref{}\thanksref{T2}\ead[label=e1]{salah@sfde.math.siu.edu}\ead[label=u1,url]{http://sfde.math.siu.edu}},
\author[B]{\fnms{Torstein K.} \snm{Nilssen}\ead[label=e2]{torsteka@math.uio.no}}\\
\and
\author[B]{\fnms{Frank N.} \snm{Proske}\ead[label=e3]{proske@math.uio.no}}
\runauthor{S.-E. A. Mohammed, T. K. Nilssen and F. N. Proske}
\affiliation{Southern Illinois University, University of Oslo and
University of Oslo}
\address[A]{S.-E. A. Mohammed\\
Department of Mathematics\\
Southern Illinois University\\
Carbondale, Illinois 62901\\
USA\\
\printead{e1}\\
\printead{u1}} 
\address[B]{T. Nilssen\\
F. Proske\\
Department of Mathematics\\
University of Oslo\\
Moltke Moes vei 35\\
P.O. Box 1053 Blindern\\
N-316 Oslo\\
Norway \\
\printead{e2} \\
\phantom{E-mail:\ }\printead*{e3}}
\end{aug}
\thankstext{T2}{Supported in part by NSF Grant DMS-07-05970 and by CMA,
Oslo, Norway.}

\received{\smonth{4} \syear{2013}}
\revised{\smonth{1} \syear{2014}}

%
\begin{abstract}
In this paper, we establish the existence of a stochastic flow of Sobolev
diffeomorphisms
\[
\mathbb{R}^{d}\ni x\quad\longmapsto\quad\phi_{s,t}(x)\in
\mathbb{R}^{d},\qquad s,t\in\mathbb{R}
\]
for a stochastic differential equation (SDE) of the form
\[
dX_{t}=b(t,X_{t}) \,dt+dB_{t},\qquad s, t\in\mathbb{R}, X_{s}=x\in\mathbb{R%
}^{d}.
\]
The above SDE is driven by a \textit{bounded measurable} drift
coefficient $%
b\dvtx\mathbb{R}\times\mathbb{R}^{d}\rightarrow\mathbb{R}^{d}$ and a
$d$-dimensional Brownian motion $B$. More specifically, we show that the
stochastic flow $\phi_{s,t}(\cdot)$ of the SDE lives in the space $%
L^{2}(\Omega;W^{1,p}(\mathbb{R}^{d},w))$ for all $s,t$ and all $p \in
(1, \infty)$, where
$W^{1,p}(\mathbb{R}^{d},w)$ denotes a weighted Sobolev space with
weight $w$
possessing a $p$th moment with respect to Lebesgue measure on $\mathbb
{R}%
^{d}$. From the viewpoint of stochastic (and deterministic) dynamical
systems, this is a striking result, since the dominant ``culture'' in
these dynamical systems is that the
flow ``inherits'' its spatial regularity from that of the driving
vector fields.

The spatial regularity of the stochastic flow yields existence and
uniqueness of a Sobolev differentiable weak solution of the
(Stratono\-vich) stochastic transport equation
\[
\cases{ \displaystyle d_{t}u(t,x)+\bigl(b(t,x)\cdot Du(t,x)\bigr)\,dt+
\sum_{i=1}^{d}e_{i}\cdot Du(t,x)
\circ dB_{t}^{i}=0,
\cr
u(0,x)=u_{0}(x),}
\]
where $b$ is \textit{bounded and measurable}, $u_{0}$ is $C_{b}^{1}$
and $%
\{e_{i}\}_{i=1}^{d}$ a basis for $\mathbb{R}^{d}$. It is well known
that the
deterministic counterpart of the above equation does not in general
have a solution.
\end{abstract}

%
\begin{keyword}[class=AMS]
\kwd{60H10}
\kwd{60H15}
\kwd{37H05}
\kwd{37H10}
\kwd{34A36}
\end{keyword}
\begin{keyword}
\kwd{SDEs with measurable coefficients}
\kwd{stochastic flows}
\kwd{Malliavin derivatives}
\kwd{Sobolev spaces}
\kwd{stochastic transport equation}
\end{keyword}

\end{frontmatter}

\setcounter{footnote}{1}
\section{Introduction}\label{sec1}

\subsection*{An overview}
This article offers the following 
novel contributions to the existing theory of stochastic
differential equations (SDEs):
\begin{itemize}
\item Well-posedness of the initial value problem for singular
SDEs driven by \textit{bounded measurable} drift vector fields and
multidimensional Brownian motion. \textit{No regularity or even continuity
hypotheses are imposed on the drift vector fields.} Furthermore, under
these hypotheses, we construct a unique stochastic flow of Sobolev
diffeomorphisms for the singular SDE.

\item The Sobolev flow of the singular SDE is employed as stochastic
characteristics in order to generate a unique Sobolev differentiable
solution to the stochastic transport equation with a \textit{bounded
measurable} drift coefficient. It is well known that the associated
\textit{deterministic} transport equation is in general ill-posed even
with a differentiable drift; cf. \cite
{Ambrosio,DiPernaLions,FlandoliGubinelliPriola}.
\end{itemize}

From a dynamical systems perspective,
the above result on singular SDEs is striking since the predominant
intuition in the current literature on stochastic (and deterministic)
dynamical systems is that the flow ``inherits'' its spatial regularity
from the driving vector fields.
Indeed, in the stochastic setting, the 
flow is in general even a \textit{little rougher} in the space variable
than the driving vector fields (\mbox{\cite{Kunita,MohammedScheutzow}}). More
specifically, it follows from work by Kunita
(\cite{Kunita}, pp. 178--179) that a SDE with $C^{k, \delta}$
coefficients ($\delta\in(0,1]$)
generates a $C^{k, \varepsilon}$ stochastic flow with positive
$\varepsilon$
\textit{strictly less} than $\delta$.
Here, the spatial $C^{k, \delta}$ regularity stands for $k$-times
differentiability with the $k$th Fr\'echet derivative $\delta$-H\"older
continuous.

In contrast with its deterministic counterpart, the singular stochastic
transport equation with multiplicative noise is well-posed due to the
regularity of the stochastic characteristics and of their occupation
measure. The latter properties have the effect of
``smoothing out'' the singularities of the drift coefficient. Needless
to say, such an effect is not available in the singular deterministic setting.


The approach in the article is probabilistic, employing ideas from the
Malliavin calculus coupled with new probabilistic estimates. In
particular, the arguments are centered around a key relative
compactness criteria for random variables developed by Nualart,
Malliavin and Da Prato. See the \hyperref[sec4]{Appendix}. The authors are not aware of any other scenarios whereby the
Malliavin calculus is employed to establish \textit{almost sure spatial}
regularity of stochastic flows for SDEs.

\subsection*{Background and statement of results}
In this article, we analyze the spatial regularity in the initial
condition $%
x \in\mathbb{R}^d$ 
for strong solutions $X_{\cdot}^{x}$ to the $d$-dimensional
SDE
%
%
\begin{equation}
X_{t}^{s,x}=x+\int_{s}^{t}b
\bigl(u,X_{u}^{s,x}\bigr)\,du+B_{t}-B_{s},\qquad
s,t\in\mathbb{R}. \label{SDE}
\end{equation}
%
In the above SDE, the drift coefficient $b\dvtx\mathbb{R}\times\mathbb
{R}%
^{d}\rightarrow\mathbb{R}^{d}$ is only \textit{Borel measurable and
bounded}, and the equation is driven by standard Brownian motion
$B_{.}$ in $%
\mathbb{R}^{d}$.

More specifically, we construct a two-parameter pathwise Sobolev differentiable
stochastic flow
\[
\mathbb{R}\times\mathbb{R}\times\mathbb{R}^{d}\ni(s,t,x)\quad\longmapsto\quad
\phi_{s,t}(x)\in\mathbb{R}^{d}
\]
for the SDE (\ref{SDE}) such that each flow map
\[
\mathbb{R}^{d}\ni x\quad\longmapsto\quad\phi_{s,t}(x)\in
\mathbb{R}^{d}
\]
is a Sobolev diffeomorphism in the sense that
%
%
\begin{equation}
\phi_{s,t}(\cdot)\mbox{ and }\phi_{s,t}^{-1}(
\cdot){\Large}\in L^{2}\bigl(\Omega,W^{1,p}\bigl(
\mathbb{R}^{d};w\bigr)\bigr) \label{Sobolev}
\end{equation}
for all $s,t\in\mathbb{R}$ and all $p \in(1, \infty)$. In (\ref
{Sobolev}) above, $W^{1,p}(%
\mathbb{R}^{d},w)$ denotes a weighted Sobolev space of mappings $\mathbb
{R}%
^{d}\rightarrow\mathbb{R}^{d}$ with any measurable weight function
$w\dvtx%
\mathbb{R}^{d}\rightarrow [0,\infty)$ satisfying the integrability
requirement
%
%
\begin{equation}
\int_{\mathbb{R}^{d}}\bigl(1+\llvert x\rrvert^{p}
\bigr)w(x)\,dx<\infty. \label{WeightCond}
\end{equation}
In particular, $\phi_{s,t}(\cdot)$ is locally $\alpha$-H\"{o}lder
continuous for all $\alpha<1$. When the SDE~(\ref{SDE}) is autonomous, we
show further that the stochastic flow corresponds to a Sobolev
differentiable perfect cocycle on $\mathbb{R}^{d}$. For precise statements
of the above results, see Theorem~\ref{MainTheorem} and Corollary \ref{CorollaryAut} in the next section.

A central objective of the article is to develop a new approach for
constructing a Sobolev differentiable stochastic flow for the SDE (\ref{SDE}).
Our approach is based
on Malliavin calculus ideas coupled with new probabilistic estimates on the
spatial weak derivatives of solutions of the SDE. A unique (and
striking) feature of these estimates is that they do not depend on the
spatial regularity of the drift coefficient~$b$.

The existence of a Sobolev differentiable stochastic flow for the SDE
(\ref{SDE}) is
exploited (Section~\ref{sec3}) to obtain a unique weak solution $u(t,x)$ of the
(Stratonovich) stochastic transport equation %
%
%
\begin{equation}
\label{StratTr} \cases{\displaystyle d_{t}u(t,x)+ \bigl(b(t,x)\cdot Du(t,x)\bigr)\,dt+
\sum_{i=1}^{d}e_{i}\cdot Du(t,x)
\circ dB_{t}^{i}=0,
\cr
u(0,x)= u_{0}(x),}
\end{equation}
when $b$ is just bounded and measurable, $u_{0}\in C_{b}^{1}(\mathbb
{R}^{d})$%
, and $\{e_{i}\}_{i=1}^{d}$ a basis for $\mathbb{R}^{d}$. This result is
interesting in view of the fact that the corresponding deterministic
transport equation
is in general ill-posed; cf. \cite{Ambrosio,DiPernaLions}.
We also note that our result holds without the existence
of the divergence of $b$, and furthermore, our solutions are spatially (and
also Malliavin) Sobolev differentiable (cf. \cite{FlandoliGubinelliPriola}).

SDEs with discontinuous coefficients and driven by Brownian motion (or more
general noise) have been an important area of study in stochastic analysis
and other related branches of mathematics. Important applications of this
class of SDEs pertain to the modeling of the dynamics of interacting
particles in statistical mechanics and the description of a variety of other
random phenomena in areas such as biology or engineering. See, for
example, \cite{Portenko} or \cite{KrylovRoeckner} and the references therein.

Using estimates of solutions of parabolic PDEs and the Yamada--Watanabe
principle, the existence of a global unique strong solution to the SDE
(\ref%
{SDE}) was first established by Zvonkin \cite{Zvonkin74} in the
one-dimensional case, when $b$ is bounded and measurable. The latter
work is
a significant development in the theory of SDEs. Subsequently,
the result was generalized by Veretennikov \cite{Veretennikov} to the
multidimensional case. More recently, Krylov and R\"{o}ckner
employed local integrability criteria on the drift coefficient $b$ to obtain
unique strong solutions of the SDE (\ref{SDE}) by using an argument of
Portenko \cite{Portenko}. An alternative approach, which does not rely
on a
pathwise uniqueness argument and which also yields the Malliavin
differentiability of solutions to (\ref{SDE}) was recently developed in
\cite%
{M-BP} and \cite{PMNPZ}. We also refer to the recent article \cite%
{DaPratoFlandoliRoeckner} for an extension of the previous results to a
Hilbert space setting. In \cite{DaPratoFlandoliRoeckner}, the authors employ
techniques based on solutions of infinite-dimensional Kolmogorov equations.

Another important issue in the study of SDEs with (bounded) measurable
coefficients is the regularity of their solutions with respect to the
initial data and the existence of stochastic flows. See \cite{Kunita,MohammedScheutzow} for more information on the existence and
regularity of stochastic
flows for SDEs, and \cite{MohammedScheutzowI,MohammedScheutzowII}
in the case of stochastic differential systems with memory.

Using the method of stochastic characteristics, stochastic flows may be employed
to prove uniqueness of solutions of stochastic transport equations
under weak
regularity hypotheses on the drift coefficient $b$. See, for example,
\cite{FlandoliGubinelliPriola}, where the authors use estimates of
solutions of
backward Kolmogorov equations to show the existence of a stochastic
flow of
diffeomorphisms with $\alpha^{\shortmid}$-H\"{o}lder continuous
derivatives for $\alpha^{\shortmid}<\alpha$, where $b\in
C([0,1];C_{b}^{\alpha}(\mathbb{R}^{d}))$, and $C_{b}^{\alpha}(\mathbb
{R}%
^{d})$ is the space of bounded $\alpha$-H\"{o}lder continuous functions.
A similar result also holds true, when $b\in L^{q}([0,1];L^{p}(%
\mathbb{R}^{d}))$ for $p,q$ such that $p\geq2,q>2,\frac{d}{p}+\frac
{2}{q}<1$%
. 
See \cite{FedrizziFlandoli}. Here, the authors construct, for any
$\alpha\in
(0,1)$, a stochastic flow of $\alpha$-H\"{o}lder continuous homeomorphisms
for the SDE (\ref{SDE}). Furthermore, it is shown in \cite{FedrizziFlandoli}
that the solution map
\[
\mathbb{R}^{d}\ni x\quad\longmapsto\quad X_{\cdot}^{x}\in
L^{p}\bigl([0,1]\times\Omega;%
\mathbb{R}^{d}
\bigr)
\]
of the SDE (\ref{SDE}) is differentiable in the $L^{p}(\Omega)$-sense for
every $p\geq2$.


The approach used in \cite{FedrizziFlandoli} 
is based on a Zvonkin-type transformation \cite{Zvonkin74} and
estimates of
solutions of an associated backward parabolic PDE. We also mention the
recent related works \cite{FedrizziFlandoli1,Fedrizzi} and \cite
{Attanasio}. For an overview of this topic, the reader may also consult the
book \cite{Flandoli}.\hskip.2pt\footnote{%
After completing the preparation of this article, personal communication
with Flandoli indicated work in preparation with Fedrizzi \cite%
{FedrizziFlandoli0} on similar issues regarding the regularity of stochastic
flows for SDEs, using a different approach.}
In this connection, it should be noted that our method for constructing a
stochastic flow for the SDE (\ref{SDE}) is heavily dependent on Malliavin
calculus ideas together with some difficult probabilistic estimates (cf.~\cite{PMNPZ}).

Our paper is organized as follows: in Section~\ref{sec2} we introduce basic
definitions and notations and provide some auxiliary results that are needed
to prove the existence of a Sobolev differentiable stochastic flow for the
SDE (\ref{SDE}). See Theorem~\ref{MainTheorem} and Corollary \ref{CorollaryAut} in Section~\ref{sec2}. 
We also briefly discuss a specific extension of this result to SDEs with
multiplicative noise. In Section~\ref{sec3}, we give an application of our
approach to
the construction of a unique Sobolev differentiable solution to the
(Stratonovich) stochastic transport equation (\ref{StratTr}). The \hyperref[sec4]{Appendix} specifies
the relative compactness
criterion of DaPrato, Malliavin and Nualart that is central to the
construction of the Sobolev flow
\cite{DaPratoMalliavinNualart}.

\section{Existence of a Sobolev differentiable stochastic flow}\label{sec2}

Throughout this paper, we denote by $%
B_{t}=(B_{t}^{(1)},\ldots,B_{t}^{(d)}), t \in\mathbb{R}$, $d$-dimensional
Brownian motion on the complete Wiener space $(\Omega,\mathcal{F},\mu)$
where $\Omega:=C(\mathbb{R};\mathbb{R}^{d})$ is given the compact open
topology and $\mathcal{F}$ is its $\mu$-completed Borel $\sigma$-field with
respect to Wiener measure $\mu$.

In order to describe the cocycle associated with the stochastic flow of our
SDE, we define the $\mu$-preserving (ergodic) Wiener shift $\theta
(t,\cdot)\dvtx \Omega\to\Omega$ by
\[
\theta(t,\omega) (s):= \omega(t+s)- \omega(t),\qquad \omega\in\Omega, t,s
\in
\mathbb{R}.
\]
The Brownian motion is then a \textit{perfect helix} with respect to
$\theta$: that is,
\[
B_{t_1 + t_2}(\omega) - B_{t_1}(\omega) = B_{t_2}\bigl(
\theta(t_1,\omega)\bigr)
\]
for all $t_1, t_2 \in\mathbb{R}$ and all $\omega\in\Omega$. The above
helix property is a convenient pathwise expression of the fact that Brownian
motion $B$ has stationary ergodic increments.


Our main focus of study in this section is the $d$-dimensional
SDE
%
%
\begin{equation}
X_{t}^{s,x}=x+\int_{s}^{t}b
\bigl(u,X_{u}^{s,x}\bigr)\,du+B_{t}-B_{s},
\qquad s, t \in\mathbb{R}, x \in\mathbb{R}^{d}, \label{FlowSDE}
\end{equation}
where the drift coefficient $b\dvtx\mathbb{R} \times\mathbb{R}%
^{d}\rightarrow\mathbb{R}^{d}$ is a bounded Borel-measurable function.

It is known that the above SDE has a unique strong global solution $%
X_{.}^{s,x}$ for each $x \in\mathbb{R}^{d}$ (\cite{Veretennikov} or
\cite{PMNPZ,M-BP}).


Here, we will establish the existence of a \textit{Sobolev-differentiable}
stochastic flow of diffeomorphisms for the SDE (\ref{FlowSDE}).

%
\begin{definition}
\label{StochasticFlow}A map $\mathbb{R} \times\mathbb{R} \times\mathbb
{R}%
^d \ni(s,t,x,\omega)\longmapsto\phi_{s,t}(x, \omega) \in\mathbb{R}^d$
is a \textit{stochastic flow of homeomorphisms} for the SDE (\ref{FlowSDE}) if there
exists a universal set $\Omega^* \in\mathcal{F}$ of full Wiener measure
such that for all $\omega\in\Omega^*$, the following statements are true:
\begin{longlist}[(iii)]
\item[(i)] For any $x\in\mathbb{R}^{d}$, the process $\phi_{s,t}(x, \omega),s,t
\in\mathbb{R}$, is a strong global solution to the SDE (\ref{FlowSDE}).

\item[(ii)] $\phi_{s,t}(x,\omega)$ is continuous in $(s,t,x) \in\mathbb{R}
\times
\mathbb{R} \times\mathbb{R}^{d}$.

\item[(iii)] $\phi_{s,t}(\cdot,\omega)=\phi_{u,t}(\cdot,\omega) \circ
\phi_{s,u}(\cdot,\omega)$ for all $s, u, t \in\mathbb{R}$.\vspace*{1pt}

\item[(iv)] $\phi_{s,s}(x,\omega)=x$ for all $x\in\mathbb{R}^{d}$ and $s \in
\mathbb{R}$.

\item[(v)] $\phi_{s,t}(\cdot,\omega)\dvtx\mathbb{R}^{d} \to\mathbb{R}^{d}$ are
homeomorphisms for all $s,t \in\mathbb{R}$.
\end{longlist}
\end{definition}

A stochastic flow $\phi_{s,t}(\cdot, \omega)$ of homeomorphisms is
said to
be \textit{Sobolev-diffe\-rentiable} if for all $s,t \in\mathbb{R}$, the maps
$\phi_{s,t}(\cdot, \omega)$ and $\phi_{s,t}^{-1}(\cdot, \omega)$ are
Sobolev-differentiable in the sense described below.

From now on, we use $\llvert\cdot\rrvert$ to denote the norm
of a
vector in $\mathbb{R}^{d}$ or a matrix in~$\mathbb{R}^{d\times d}$.

In order to prove the existence of a Sobolev differentiable flow for
the SDE
(\ref{FlowSDE}), we need to introduce a suitable class of weighted Sobolev
spaces. Fix $p\in(1,\infty)$ and let $w\dvtx\mathbb
{R}^{d}\rightarrow
(0,\infty)$ be a Borel-measurable function satisfying
%
%
\begin{equation}
\int_{\mathbb{R}^{d}}\bigl(1+\llvert x\rrvert^{p}
\bigr)w(x)\,dx<\infty. \label{WeightCond}
\end{equation}

Let $L^{p}(\mathbb{R}^{d},w)$ denote the Banach space of all Borel-measurable
functions $u=(u_{1},\ldots,u_{d})\dvtx\mathbb{R}^{d}\rightarrow
\mathbb{R}^{d}$
such that
%
%
\begin{equation}
\int_{\mathbb{R}^{d}}\bigl\llvert u(x)\bigr\rrvert^{p}w(x)\,dx<
\infty \label{LPW}
\end{equation}
and equipped with the norm
\[
\|u\|_{L^p(\mathbb{R}^{d},w)}:= \biggl[\int_{\mathbb{R}^{d}}\bigl
\llvert u(x)
\bigr\rrvert^{p}w(x)\,dx \biggr]^{1/p}. 
\]
Furthermore, denote by $W^{1,p}(\mathbb{R}^{d},w)$ the linear space of
functions $u\in L^{p}(\mathbb{R}^{d},w)$ with weak partial derivatives $
D_{j}u \in L^{p}(\mathbb{R}^{d},w)$ for $j=1,\ldots,d$. We equip this space
with the complete norm
%
%
\begin{equation}
\llVert u\rrVert_{1,p,w}:=\llVert u\rrVert_{L^{p}(\mathbb{R}%
^{d},w)}+\sum
_{i,j=1}^{d}\llVert D_{j}u_{i}
\rrVert_{L^{p}(\mathbb{R}%
^{d},w)}. \label{WeightedSobolevNorm}
\end{equation}

We will show that the strong solution $X_{t}^{s,.}$ of the SDE (\ref
{FlowSDE}%
) is in $L^{2}(\Omega,\break L^{p}(\mathbb{R}^{d},w))$ when $p \in(1, \infty
)$ (see Corollary %
\ref{Holder}). In fact, the SDE (\ref{FlowSDE}) \mbox{implies} the following
estimate:
\[
\bigl\llvert X_{t}^{s,x}\bigr\rrvert^{p}\leq
c_{p}\bigl(\llvert x\rrvert^{p}+\llvert t-s\rrvert
^{p}\llVert b\rrVert_{\infty
}^{p}+\llvert
B_{t}-B_{s}\rrvert^{p}\bigr)
\]
for all $s,t \in\mathbb{R}, x \in\mathbb{R}^d$.

On the other hand, it is easy to see 
that the solutions $X_{t}^{s,.}$ of SDE (\ref{FlowSDE}) are in general
not in $L^{p}(\mathbb{R}^{d},dx)$
with respect to Lebesgue measure $dx$ on $\mathbb{R}^d$: just consider the
special trivial case $b \equiv0$. This implies that solutions of the
SDE (%
\ref{FlowSDE}) (if they exist) may not belong to the Sobolev space
$W^{1,p}(%
\mathbb{R}^{d},dx), p \in(1, \infty)$. However, we will show that such
solutions do indeed
belong to the weighted Sobolev spaces $W^{1,p}(\mathbb{R}^{d},w)$ for
$p \in(1, \infty)$.

%
\begin{remark}
\label{Muckenhoupt}(i) Let $w\dvtx\mathbb{R}^{d}\rightarrow
(0,\infty)$ be a weight function in Muckenhoupt's $A_{p}$-class ($%
1<p<\infty$), that is a locally (Lebesgue) integrable function on
$\mathbb{R%
}^{d}$ such that
\[
\sup\biggl( \frac{1}{\lambda_{d}(B)}\int_{B}w(x)\,dx \biggr)
\biggl( \frac
{1}{%
\lambda_{d}(B)}\int_{B}\bigl(w(x)
\bigr)^{1/(1-p)}\,dx \biggr) ^{p-1}=:c_{w,p}<\infty,
\]
where the supremum is taken over all balls $B$ in $\mathbb{R}^{d}$ and $
\lambda_{d}$ is Lebesgue measure on~$\mathbb{R}^{d}$. For example, the
function $w(x)=\llvert x\rrvert^{\gamma}$ is an $A_{p}$-weight iff $%
-d<\gamma<d(p-1)$. Other examples of weights are given by positive
superharmonic functions. See, for example, \cite{HKM} and \cite{Kufner}
and the
references therein. 
Denote by $H^{1,p}(\mathbb{R}^{d},w)$ the completion of $C^{\infty}(%
\mathbb{R}^{d})$ with respect to the norm $\llVert\cdot\rrVert_{1,p,w}$
in (\ref{WeightedSobolevNorm}). If $w$ is a $A_{p}$-weight, then we have
\[
W^{1,p}\bigl(\mathbb{R}^{d},w\bigr)=H^{1,p}\bigl(
\mathbb{R}^{d},w\bigr)
\]
for all $1<p<\infty$; see, for example, \cite{HKM}.

(ii) Let $p_{0}=\inf\{q>1\dvtx w$ is a $A_{q}$-weight\} and let $u\in
W^{1,p}(%
\mathbb{R}^{d},w)$. If $p_{0}<p/d$, then $u$ is locally H\"{o}lder
continuous with any exponent $\alpha$ such that $0<\alpha<1-dp_{0}/p$.
\end{remark}

We now state our main result in this section which gives the existence
of a
Sobolev differentiable stochastic flow for the SDE (\ref{FlowSDE}).

%
\begin{theorem}\label{MainTheorem}
In the SDE (\ref{FlowSDE}), assume that the drift
coefficient
$b$ is Borel-measurable and bounded. Then the SDE (\ref{FlowSDE}) has a
Sobolev differentiable stochastic flow $\phi_{s,t}\dvtx \mathbb{R}^d
\to
\mathbb{R}^d, s,t \in\mathbb{R}$: that is,
%
\[
\phi_{s,t}(\cdot)\mbox{ and }\phi_{s,t}^{-1}(
\cdot){\Large}\in L^{2}\bigl(\Omega,W^{1,p}\bigl(
\mathbb{R}^{d},w\bigr)\bigr)
\]
for all $s,t\in\mathbb{R}$ and all $p \in(1, \infty)$.
\end{theorem}

%
\begin{remark}
\label{HolderProperty}If $w$ is a $A_{p}$-weight, then it follows from
Remark %
\ref{Muckenhoupt}(ii) that a version of $\phi_{s,t}(\cdot)$ is
locally H\"{o}lder continuous for all $0<\alpha<1$ and all $s,t$.
\end{remark}

The following corollary is a consequence of Theorem~\ref{MainTheorem} and the helix property of the Brownian motion.

%
\begin{corollary}
\label{CorollaryAut}Consider the autonomous SDE
%
%
\begin{equation}
X_{t}^{s,x}=x+\int_{s}^{t}b
\bigl(X_{u}^{s,x}\bigr)\,du+B_{t}-B_{s},\qquad
s,t\in\mathbb{R} \label{SDEAut}
\end{equation}
with bounded Borel-measurable drift $b\dvtx\mathbb{R}^{d}\rightarrow
\mathbb
{R}%
^{d}$. Then the stochastic flow of the SDE (\ref{SDEAut}) has a version
which generates a perfect Sobolev-differentiable cocycle $(\phi
_{0,
t},\theta(t,\cdot))$ where $\theta(t,\cdot)\dvtx\Omega\rightarrow
\Omega$ is the $\mu$-preserving Wiener shift. More specifically, the
following perfect cocycle property holds \textit{for all} $\omega\in
\Omega$ and
all $t_{1},t_{2}\in\mathbb{R}$:
\[
\phi_{0,t_{1}+t_{2}}(\cdot,\omega)=\phi_{0,t_{2}}\bigl(\cdot,\theta
(t_{1},\omega)\bigr)\circ\phi_{0,t_{1}}(\cdot,\omega).
\]
\end{corollary}

We will prove Theorem~\ref{MainTheorem} through a sequence of lemmas and
propositions. We begin by stating our main proposition.

%
\begin{proposition}
\label{MainProposition} Let $b\dvtx\mathbb{R} \times\mathbb{R}^d
\rightarrow
\mathbb{R}^d$ be bounded and measurable. Let $U$ be an open and bounded
subset of $\mathbb{R}^d$. For each $t \in\mathbb{R}$ and $p > 1$, we have
\[
X_t^{\cdot} \in L^2\bigl(\Omega;
W^{1,p}(U)\bigr).
\]
\end{proposition}

We will prove Proposition~\ref{MainProposition} using two steps. In the
\textit{first step}, we show that for a bounded smooth function $%
b\dvtx[0,1]\times\mathbb{R}^{d}\rightarrow\mathbb{R}^{d}$ with compact
support, it is possible to estimate the norm of $X_{t}^{\cdot}$ in $%
L^{2}(\Omega,W^{1,p}(U))$ independently of the size of the spatial
Jacobian $b^{\prime}$ of
$b$, with
the estimate depending only on $\Vert b\Vert_{\infty}$.

In the \textit{second step}, we will approximate our bounded measurable
coefficient $b$ by 
a sequence $\{b_n\}_{n=1}^{\infty}$ of smooth compactly supported functions
as in step 1. We then show that the corresponding sequence $X_t^{n,\cdot}$
of solutions is relatively compact in $L^2(\Omega)$ when integrated against
a test function on $\mathbb{R}^d$. By step 1, we use weak compactness
of the
above sequence in $L^2(\Omega, W^{1,p}(U))$ to conclude that the limit point
$X_t^{\cdot}$ of the above sequence must also lie in this space.

We now turn to the first step of our procedure. Note that if $b$ is a
compactly supported smooth function, the corresponding solution\vspace*{1pt} of the
SDE (%
\ref{SDE}) is (strongly) differentiable with respect to $x$, and the
first-order spatial Jacobian $\frac{\partial}{\partial x}X_{t}^{x}$
satisfies the
linearized random ODE
%
%
\begin{equation}
\cases{ \displaystyle d\frac{\partial}{\partial x}X_{t}^{x}=
b^{\prime}\bigl(t,X_{t}^{x}\bigr)
\frac{
\partial}{\partial x}X_{t}^{x}\,dt,
\vspace*{5pt}\cr
\displaystyle
\frac{\partial}{\partial x}X_{0}^{x}= \mathcal{I}_{d}.}\label{linearizedeq}
\end{equation}
In the above equation and throughout this section, $\mathcal{I}_{d}$ is
the $d\times d$ identity matrix and $b^{\prime}(t,x):= ( \frac
{\partial}{\partial x_{i}}b^{(j)}(t,x) ) _{1\leq
i,j\leq d}$ denotes the spatial Jacobian derivative of $b$.

A key estimate in the first step of the argument is provided by the
following proposition.

%
\begin{proposition}
\label{sobolevestimate} Assume that $b$ is a smooth function with
compact support. Then for any $p \in[1, \infty)$ and $t \in\mathbb
{R}$, we have
the following estimate for the solution of the linearized equation (\ref
{linearizedeq}):
\[
\sup_{x \in\mathbb{R}^d} E \biggl[ \biggl| \frac{\partial}{\partial x} X_t^x
\biggr|^p \biggr] \leq C_{d,p}\bigl(\|b\|_{\infty}\bigr),
\]
where $C_{d,p} $ is an increasing continuous function depending only on $d$
and $p$.
\end{proposition}

The proof of Proposition~\ref{sobolevestimate} relies on the following
sequence of lemmas which provide estimates on expressions depending on
the Gaussian distribution and its derivatives. To this end, we define
$P(t,z):= (2\pi t)^{d/2}e^{-|z|^2/2t}, t > 0$, where $|z|$ is the
Euclidean norm of a vector $z \in\mathbb{R}^d$.

%
\begin{lemma}\label{le8}
Let $\phi, h \dvtx [0,1] \times\mathbb{R}^d \rightarrow\mathbb{R}$ be
measurable functions such that $| \phi(s,z)| \leq e^{-|z|^2 /3s}$ and
$\|
h \|_{\infty} \leq1$. Also let $\alpha,\beta\in\{0,1\}^d$ be multiindices
such that $|\alpha| = |\beta| = 1$. Then there exists a universal
constant $%
C $ (independent of $\phi$, $h$, $\alpha$ and $\beta$) such that
\[
\biggl\llvert\int_{1/2}^1 \int
_{t/2}^t \int_{\mathbb{R}^d} \int
_{\mathbb{R}^d} \phi(s,z) h(t,y) D^{\alpha} D^{\beta}
P(t-s, y -z) \,dy \,dz \,ds \,dt \biggr\rrvert\leq C.
\]
Furthermore, there is a universal positive constant (also denoted by)
$C$ such that for measurable functions $g$ and $h$ bounded by 1, we have
\[
\biggl\llvert\int_{1/2}^1 \int
_{t/2}^t \int_{\mathbb{R}^d} \int
_{\mathbb{R}^d} g(s,z)P(s,z) h(t,y) D^{\alpha}D^{\beta}P(t-s,y-z)
\,dy\,dz\,ds\,dt \biggr\rrvert\leq C
\]
and
\[
\biggl\llvert\int_{1/2}^1 \int
_{t/2}^t \int_{\mathbb{R}^d} \int
_{\mathbb{R}^d} g(s,z)D^{\gamma}P(s,z) h(t,y)
D^{\alpha}D^{\beta}P(t-s,y-z) \,dy\,dz\,ds\,dt \biggr\rrvert\leq C.
\]
\end{lemma}

\begin{pf}
We will only give a proof of the first estimate in the lemma. The
proofs of the second and third estimates are left to the reader.

Denote the first integral in the lemma by $I$. Let $l,m \in\mathbb
{Z}^d$ and define $[l, l+1):= [l^{(1)}, l^{(1)} +1) \times\cdots
\times[l^{(d)}, l^{(d)} +1)$ and similarly for $[m,m+1)$. Truncate the
functions $\phi, h$ by setting $\phi_l(s,z):= \phi(s,z) 1_{[l,
l+1)}(z)$ and $h_m(t,y):= h(t,y) 1_{[m,m+1)}(y)$.

In the first integral, we replace
$\phi$, $h$ by $\phi_l$, $h_m$, respectively, and thus define
\[
I_{l,m}:= \int_{1/2}^1 \int
_{t/2}^t \int_{\mathbb{R}^d} \int
_{\mathbb{R}^d} \phi_l(s,z) h_m(t,y)
D^{\alpha} D^{\beta} P(t-s, y -z) \,dy \,dz \,ds \,dt.
\]
Therefore, we can write $I = \sum_{l,m \in\mathbb{Z}^d} I_{l,m}$.
Below, we let $C$ be a generic constant that may vary from line to line.

Assume $\| l - m \|_{\infty}:= \max_{i}|l^{(i)} - m^{(i)}| \geq2$.
For $z \in[l, l+1)$ and $y \in[m, m+1)$ we have $| z - y | \geq\|l -
m\|_{\infty} - 1$. If $\alpha\neq\beta$, we have that
\[
D^{\alpha} D^{\beta} P(t-s, z-y) = \frac{(z^{(i)} - y^{(i)})(z^{(j)} -
y^{(j)})}{(t-s)^2}P(t-s,y-z)
\]
for a suitable choice of $i,j$. Then we can find $C$ such that
\[
\bigl|D^{\alpha} D^{\beta} P(t-s, z-y)\bigr| \leq C e^{-(\|l - m\|_{\infty}- 2)^2/4}.
\]
If $\alpha= \beta$, we have
\[
\bigl(D^{\alpha}\bigr)^2 P(t-s, y-z) = \biggl(
\frac{(y^{(i)} - z^{(i)})^2}{t-s} - 1 \biggr)\frac{P(t-s, y-z)}{t-s} %
\]
and similarly we find $C$ such that
\[
\bigl|\bigl(D^{\alpha}\bigr)^2 P(t-s, y-z)\bigr| \leq C e^{-(\|l -m \|_{\infty} - 2)^2/4}. %
\]
In both cases we have $| I_{l,m} | \leq C e^{-|l|^2/8}e^{-(\| l-m\|
_{\infty} - 2)^2/4}$ and it follows that
\[
\sum_{\|l-m\|_{\infty} \geq2} |I_{l,m}| \leq C.  %
\]

Assume $\| l -m \|_{\infty} \leq1$ and let $\hat{\phi}_l(s,u)$ and
$\hat{h}_m(t,u)$ be the Fourier transforms in the second variable,
defined by
\[
\hat{h}_{m}(t,u):= (2 \pi)^{-d/2} \int_{\mathbb{R}^d}
h(t,x) e^{-i
(u,x)} \,dx %
\]
and similarly for $\hat{\phi}_l(s,u)$. By the Plancherel theorem we
have that
\[
\int_{\mathbb{R}^d} \hat{\phi}_l(s,u)^2 \,du
= \int_{\mathbb{R}^d} \phi_l(s,z)^2 \,dz \leq
C e^{-|l|^2 / 6} %
\]
for all $s \in[0,1]$ and
\[
\int_{\mathbb{R}^d} \hat{h}_m(t,u)^2 \,du =
\int_{\mathbb{R}^d} h_m(t,y)^2 \,dy \leq1.
\]
We can write
%
%
\begin{equation}
\label{2to1integral} \qquad I_{l,m} = \int_{1/2}^1
\int_{t/2}^t \int_{\mathbb{R}^d} \hat{
\phi}_l(s,u) \hat{h}_m(t,-u) u^{(i)}u^{(j)}(t-s)
e^{-(t-s)|u|^2/2}\,du \,ds \,dt.
\end{equation}

To see this, start with the right-hand side. Then we have by Fubini's theorem
\begin{eqnarray*}
&& \int_{\mathbb{R}^d} \hat{h}_m(t,-u) \hat{
\phi}_l(s,u) u^i u^j(t-s) e^{-(t-s) |u |^2 /2}
\,du
\\
&&\qquad  = (2 \pi)^{-d} \int_{\mathbb{R}^d} \int
_{\mathbb{R}^d} \int_{\mathbb
{R}^d} h_m(t,x)
e^{i(u,x)} \phi_l(s,y) e^{-i(u,y)} u^i
u^j(t-s)
\\
&&\hspace*{117pt}{}\times  e^{-(t-s) |u|^2/2} \,du \,dx \,dy
\\
&&\qquad  = \int_{\mathbb{R}^d} \int_{\mathbb{R}^d}
h_m(t,x) \phi_l(s,y) (t-s)
\\
&&\hspace*{65pt}{}\times  \biggl[(2
\pi)^{-d} \int_{\mathbb{R}^d} e^{i(u,x-y)}
u^iu^j e^{-(t-s)
|u|^2/2} \,du \biggr] \,dx \,dy.
\end{eqnarray*}

Now look at the expression in the square brackets. Substitute $v =
\sqrt{t-s} u$ to get
\begin{eqnarray*}
&& (2 \pi)^{-d} \int_{\mathbb{R}^d} e^{i(u,x-y)}
u^iu^j e^{-(t-s)
|u|^2/2} \,du
\\
&&\qquad = (2 \pi)^{-d} (t-s)^{-d/2} \int_{\mathbb{R}^d}
e^{i(v/{\sqrt {t-s}}, x-y)} \frac{v^i}{\sqrt{t-s}} \frac{v^j}{\sqrt{t-s}} e^{-|v|^2
/2} \,dv
\\
&&\qquad = (2 \pi)^{-d} (t-s)^{-d/2} (t-s)^{-1} \int
_{\mathbb{R}^d} e^{i(v,
(x-y)/{\sqrt{t-s}})} v^i v^j
e^{-|v|^2 /2} \,dv.
\end{eqnarray*}
Now\vspace*{1pt} put $f(v) = e^{-|v|^2 /2}$ and $p(v) = v^{(i)}v^{(j)}$. From
properties of the Fourier transform, we know that $\widehat{pf} =
D^{\alpha} D^{\beta} \hat{f} $ and $\hat{f} = f$. This gives that the
above expression is equal to
\begin{eqnarray*}
&& (2 \pi)^{-d/2}(t-s)^{-d/2}(t-s)^{-1}D^{\alpha}
D^{\beta} f \biggl( \frac
{x-y}{\sqrt{t-s}} \biggr)
\\
&&\qquad  = (t-s)^{-1}
D^{\alpha} D^{\beta} P(t-s, x-y).
\end{eqnarray*}
This proves equation (\ref{2to1integral}).

Applying the inequality $ab \leq\frac{1}{2} a^2c + \frac{1}{2}
b^2c^{-1}$ to (\ref{frstDerEst}) with $a = \hat{\phi}_l(s,u)u^{(i)}$, $b = \hat
{h}_m(t,-u)u^{(j)}$ and $c = e^{|l|^2/12}$ we get
\begin{eqnarray*}
| I_{l,m} | & \leq&\frac{1}{2} \int_{1/2}^1
\int_{t/2}^t \int_{\mathbb
{R}^d} \hat{
\phi}_l(s,u)^2\bigl(u^{(i)}\bigr)^2
e^{|l|^2/12} e^{-(t-s)|u|^2/2} \,du \,ds \,dt
\\
&&{} + \frac{1}{2} \int_{1/2}^1 \int
_{t/2}^t \int_{\mathbb{R}^d} \hat
{h}_m(t,-u)^2\bigl(u^{(j)}\bigr)^2
e^{-|l|^2/12} e^{-(t-s)|u|^2/2} \,du \,ds \,dt
\\
& \leq&\frac{1}{2} \int_{1/2}^1 \int
_{t/2}^t \int_{\mathbb{R}^d} \hat{
\phi}_l(s,u)^2|u|^2 e^{|l|^2/12}
e^{-(t-s)|u|^2/2} \,du \,ds \,dt
\\
&&{} + \frac{1}{2} \int_{1/2}^1 \int
_{t/2}^t \int_{\mathbb{R}^d} \hat
{h}_m(t,-u)^2|u|^2 e^{-|l|^2/12}
e^{-(t-s)|u|^2/2} \,du \,ds \,dt.
\end{eqnarray*}
For the first term, integrate first with respect to $t$ in order to get
\[
\int_{1/2}^1 \int_{t/2}^t
\int_{\mathbb{R}^d} \hat{\phi}_l(s,u)^2|u|^2
e^{|l|^2/12} e^{-(t-s)|u|^2/2} \,du \,ds \,dt \leq C e^{-|l|^2/12} %
\]
and for the second term, integrate with respect to $s$ first to get
\[
\int_{1/2}^1 \int_{t/2}^t
\int_{\mathbb{R}^d} \hat{h}_m(t,-u)^2|u|^2
e^{-|l|^2/12} e^{-(t-s)|u|^2/2} \,du \,ds \,dt \leq C e^{-|l|^2/12},
\]
which gives $ | I_{l,m}| \leq C e^{-|l|^2/12}$, and hence
\[
\sum_{\|l-m\|_{\infty} \leq1} |I_{l,m}| \leq C.  %
\]\upqed
\end{pf}


Using the previous lemma we can show the following:

%
\begin{lemma}
\label{cnvLmm}
Let $g,h\dvtx [0,1] \times\mathbb{R}^d \rightarrow\mathbb{R}$ be
Borel-measurable and bounded by 1 and let $r \geq0$. As before we let
$\alpha, \beta, \gamma$ be multiindices with length 1. Then there exists a
universal constant $C$ such that
\begin{eqnarray*}
&& \biggl\llvert\int_{t_0}^t \int
_{t_0}^{t_1} \int_{\mathbb{R}^d} \int
_{\mathbb{R}^d} g(t_2,z)P(t_2-t_0,z)
h(t_1,y)
\\
&&\hspace*{68pt}{}\times D^{\alpha}D^{\beta}P(t_1-t_2,y-z)
(t-t_1)^r \,dy\,dz\,dt_2\,dt_1 \biggr
\rrvert
\\
&&\qquad \leq C(1+r)^{-1} (t-t_0)^{r+1}
\end{eqnarray*}
and
\begin{eqnarray*}
&& \biggl\llvert\int_{t_0}^t \int
_{t_0}^{t_1} \int_{\mathbb{R}^d} \int
_{\mathbb{R}^d} g(t_2,z)D^{\gamma}P(t_2-t_0,z)
h(t_1,y)
\\
&&\hspace*{70pt}\times{}D^{\alpha} D^{\beta}P(t_1-t_2,y-z)
(t-t_1)^r \,dy\,dz\,dt_2\,dt_1 \biggr
\rrvert
\\
&&\qquad \leq C(1+r)^{-1/2} (t-t_0)^{r+1/2}.
\end{eqnarray*}
\end{lemma}

\begin{pf}
We begin by proving the first estimate in the lemma for $t = 1$, $t_0 =
0$. The following estimate holds for
each integer $k\ge0$:
\begin{eqnarray*}
&& \biggl\llvert\int_{2^{-k-1}}^{2^{-k}} \int
_{t/2}^t \int_{\mathbb{R}^d} \int
_{\mathbb{R}^d} g(s,z)P(s,z) h(t,y)
\\
&&\hspace*{83pt}{}\times D^{\alpha}D^{\beta}P(t-s,y-z)
(1-t)^r \,dy\,dz\,ds\,dt \biggr\rrvert
\\
&&\qquad \leq C\bigl(1-2^{-k-1}\bigr)^r 2^{-k}.
\end{eqnarray*}
To see this, use the fact $P(at,z) = a^{-d/2}P(t,a^{-1/2}z)$ and make
the following substitutions in the second estimate in Lemma~\ref{le8}: $t':=
2^k t$ and $s':= 2^k s$,
$z':= 2^{k/2}z$ and $y':= 2^{k/2}y$, $ \tilde{h}(t,y):= \frac
{(1-t)^r}{(1-2^{-k-1})^r} h(t,y) $.

Summing the above inequalities over $k$ gives
\begin{eqnarray*}
&& \biggl\llvert\int_{0}^1 \int
_{t/2}^t \int_{\mathbb{R}^d} \int
_{\mathbb{R}^d} g(s,z)P(s,z) h(t,y)D^{\alpha}D^{\beta}P(t-s,y-z)
(1-t)^r \,dy\,dz\,ds\,dt \biggr\rrvert
\\
&&\qquad \leq C (1+r)^{-1}.
\end{eqnarray*}
Moreover, it is easy to see that
%
\begin{eqnarray*}
&& \biggl\llvert\int_0^1 \int
_0^{t/2} \int_{\mathbb{R}^d} \int
_{\mathbb{R}^d} g(s,z) P(s,z) h(t,y)D^{\alpha}D^{\beta}
P(t-s,y-z) (1-t)^r \,dy\,dz\,ds\,dt \biggr\rrvert
\\
&&\qquad \leq C \int_0^1 \int_0^{t/2}
(t-s)^{-1}(1-t)^r \,ds \,dt \leq C(1+r)^{-1}
\end{eqnarray*}
and combining these bounds gives the first assertion of the lemma for
$t=1$, $t_0 = 0$. For general $t$ and $t_0$ use the change of variables
$t_1':= \frac{t_1 - t_0}{t-t_0}$, $t_2':= \frac{t_2 - t_0}{t - t_0}$,
$y':= (t-t_0)^{-1/2}y$ and $z':= (t-t_0)^{-1/2}z$.

The second assertion of the lemma is proved similarly.
\end{pf}

We now turn to the following key estimate (cf.
\cite{Da07}, Proposition 2.2):

%
\begin{lemma}
\label{oldestimate} Let $B$ be a $d$-dimensional Brownian Motion starting
from the origin and $b_1, \ldots, b_n$ be compactly supported continuously
differentiable functions $b_i \dvtx [0,1] \times\mathbb{R}^d
\rightarrow
\mathbb{R}$ for $i=1,2, \ldots, n$. Let $\alpha_i\in\{0,1\}^d$ be a
multiindex such that $| \alpha_i| = 1$ for $i = 1,2, \ldots, n$. Then there
exists a universal constant $C$ (independent of $\{ b_i \}_i$, $n$, and
$\{
\alpha_i \}_i$) such that
%
\begin{eqnarray}\label{estimate}
&& \Biggl\llvert E \Biggl[ \int_{t_0 < t_1 < \cdots< t_n < t} \Biggl( \prod
_{i=1}^n D^{\alpha_i}b_i(t_i,x+B_{t_i})
\Biggr) \,dt_1 \cdots \,dt_n \Biggr] \Biggr\rrvert
\nonumber\\[-8pt]\\[-8pt]
&&\qquad \leq\frac{C^n \prod_{i=1}^n \|b_i \|_{\infty}
(t-t_0)^{n/2}}{\Gamma
((n/2) + 1)},\nonumber
\end{eqnarray}
where $\Gamma$ is the Gamma-function and $x \in\mathbb{R}^d$. Here, $%
D^{\alpha_i}$ denotes the partial derivative with respect to the
$j^{\prime}$th space variable, where $j$ is the position of the $1$ in~$\alpha_i$.
\end{lemma}

\begin{pf}
Without loss of generality, assume that $\| b_i \|_{\infty} \leq1$ for
$i =\break
1, 2, \ldots, n$. Using the Gaussian density, we write the left-hand side
of the estimate %
\eqref{estimate} in the form
\begin{eqnarray*}
&& \Biggl\llvert\int_{t_0 < t_1 < \cdots< t_n < t} \int_{\mathbb{R}^{dn}}
\prod_{i=1}^n D^{\alpha_i}
b_i(t_i,x+ z_i)
\\
&&\hspace*{99pt}{}\times  P(t_i -
t_{i-1}, z_i - z_{i-1} ) \,dz_1 \cdots \,dz_n \,dt_1 \cdots \,dt_n \Biggr\rrvert.
\end{eqnarray*}
Introduce the notation
\begin{eqnarray*}
&& J_n^{\alpha} (t_0,t, z_0)
\\
&&\qquad = \int
_{ t_0 < t_1 < \cdots< t_n < t} \int_{%
\mathbb{R}^{dn}} \prod
_{i=1}^n D^{\alpha_i} b_i(t_i,x+z_i)
\\
&&\hspace*{129pt}{}\times
P(t_i - t_{i-1}, z_i - z_{i-1} )
\,dz_1 \cdots \,dz_n \,dt_1 \cdots \,dt_n,
\end{eqnarray*}
where $\alpha= (\alpha_1, \ldots,\alpha_n) \in\{0,1 \}^{nd}$. We shall show
that $|J_n^{\alpha}(t_0,t,z_0)| \leq C^n(t-t_0)^{n/2} / \Gamma( n/2 + 1)$,
thus proving the proposition.

To do this, we will use integration by parts to shift the derivatives
from the $b_i$'s onto
the Gaussian kernel. This will be done by introducing the alphabet
\[
\mathcal{A}(\alpha) = \bigl\{ P, D^{\alpha_1}P, \ldots, D^{\alpha_n}P,
D^{\alpha_1}D^{\alpha_2}P, \ldots, D^{\alpha_{n-1}}D^{\alpha_n}P \bigr
\},
\]
where $D^{\alpha_i}$, $D^{\alpha_i}D^{\alpha_{i+1}}$ denotes the derivatives
in $z$ of $P(t,z)$.

Take a string $S = S_1 \cdots S_n$ in $\mathcal{A}(\alpha)$ and define
\begin{eqnarray*}
&& I_S^{\alpha} (t_0,t,z_0)
\\
&&\qquad = \int
_{t_0 < \cdots< t_n <t} \int_{\mathbb{R}
^{dn}} \prod
_{i=1}^n b_i(t_i,
x+z_i)
\\
&&\hspace*{115pt}{}\times  S_i(t_i - t_{i-1},
z_i - z_{i-1}) \,dz_1 \cdots \,dz_n
\,dt_1 \cdots \,dt_n.
\end{eqnarray*}
We will need only a special type of strings: say that a string is \emph{allowed} if, when all the $D^{\alpha_i}P$'s are removed from the
string, a
string of the form $P \cdot D^{\alpha_s}D^{\alpha_{s+1}} P \cdot P \cdot
D^{\alpha_{s+1}}D^{\alpha_{s+2}} P \cdots P \cdot
D^{\alpha_r}D^{\alpha_{r+1}} P$ for $s \geq1$, $r \leq n-1$ remains. Also,
we will require that the first derivatives $D^{\alpha_i}P$ are written
in an
increasing order with respect to $i$.

We now claim that
%
\[
J_n^{\alpha}(t_0,t,z_0) = \sum
_{j=1}^{2^{n-1}} \varepsilon_j
I_{S^j}^{\alpha}(t_0,t,z_0),
\]
where each $\varepsilon_j$ is either $-1$ or $1$ and each $S^j$ is an allowed
string in $\mathcal{A}(\alpha)$.
To see this, we proceed by induction on $n \geq1$.

The claim obviously holds for $n=1$. Assume that it holds for $n\geq
1$, and let $b_{0}$ be another function satisfying the requirements of
the lemma. Likewise with $\alpha_{0}$. Then
\begin{eqnarray*}
&& J_{n+1}^{(\alpha_{0},\alpha)}(t_{0},t,z_{0})
\\
&&\qquad = \int
_{t_{0}}^{t}\int_{%
\mathbb{R}^{d}}D^{\alpha
_{0}}b_{0}(t_{1},x+z_{1})P(t_{1}-t_{0},z_{1}-z_{0})J_{n}^{\alpha
}(t_{1},t,z_{1})\,dz_{1}\,dt_{1}
\\
&&\qquad = -\int_{t_{0}}^{t}\int_{\mathbb{R}^{d}}b_{0}(t_{1},x+z_{1})D^{\alpha
_{0}}P(t_{1}-t_{0},z_{1}-z_{0})J_{n}^{\alpha
}(t_{1},t,z_{1})\,dz_{1}\,dt_{1}
\\
&&\quad\qquad{} -\int_{t_{0}}^{t}\int_{\mathbb{R}%
^{d}}b_{0}(t_{1},x+z_{1})P(t_{1}-t_{0},z_{1}-z_{0})D^{\alpha
_{0}}J_{n}^{\alpha}(t_{1},t,z_{1})\,dz_{1}\,dt_{1}.
\end{eqnarray*}
Notice that
\[
D^{\alpha_{0}}I_{S}^{\alpha}(t_{1},t,z_{1})=-I_{\widetilde{S}}^{(\alpha
_{0},\alpha)}(t_{1},t,z_{1}),
\]
where
\[
\widetilde{S}=\cases{ D^{\alpha_{0}}P\cdot S_{2}\cdots
S_{n} &\quad if $S=P\cdot S_{2}\cdots S_{n}$,
\vspace*{3pt}\cr
D^{\alpha_{0}}D^{\alpha_{1}}P\cdot S_{2}\cdots
S_{n} &\quad if $S=D^{\alpha_{1}}P\cdot S_{2}\cdots
S_{n}$.}
\]
Here, $\widetilde{S}$ is not an allowed string in $\mathcal{A}(\alpha)$. So
from the induction hypothesis $D^{\alpha_{0}}J_{n}^{\alpha
}(t_{0},t,z_{0})=\sum_{j=1}^{2^{n-1}}-\varepsilon_{j}I_{\widetilde{S}}^{(\alpha
_{0},\alpha)}(t_{0},t,z_{0})$. This gives
\[
J_{n+1}^{(\alpha_{0},\alpha)}=\sum_{j=1}^{2^{n-1}}-
\varepsilon_{j}I_{D^{\alpha_{0}}P\cdot S^{j}}^{(\alpha_{0},\alpha
)}+\sum
_{j=1}^{2^{n-1}}\varepsilon_{j}I_{P\cdot\widetilde{S}^{j}}.
\]
It is easily checked that when $S^{j}$ is an allowed string in $\mathcal
{A}%
(\alpha)$, both $D^{\alpha_{0}}P\cdot S^{j}$ and $P\cdot\widetilde{S}^{j}$
are allowed strings in $\mathcal{A}(\alpha_{0},\alpha)$.

This proves the claim.

For the rest of the proof of Lemma~\ref{oldestimate} we will bound $%
I_S^{\alpha}$ when $S$ is an allowed string; that is, we will show that
there is a positive constant $M$ such that
\[
I_S^{\alpha}(t_0,t,z_0) \leq
\frac{M^n (t-t_0)^{n/2}}{\Gamma((n/2) +
1)}
\]
for all integers $n \geq1$ and for each allowed string $S$ in the
alphabet $%
\mathcal{A}(\alpha)$.

We proceed by induction on $n \geq0$: the case $n=0$ is immediate, so
assume $n > 0$ and that this holds for
all allowed strings of length less than $n$.
We consider the three cases:
\begin{longlist}[(3)]
\item[(1)] $S = D^{\alpha_1}P \cdot S'$ where $S'$ is a string in $\mathcal
{A}(\alpha')$ and $\alpha':= (\alpha_2, \ldots, \alpha_n)$;

\item[(2)]
$S = P \cdot D^{\alpha_1} D^{\alpha_2} P \cdot S'$ where $S'$ is a
string in $\mathcal{A}(\alpha')$ and $\alpha': = (\alpha_3, \ldots,
\alpha_n)$;

\item[(3)]
$S = P \cdot D^{\alpha_1}P \cdots D^{\alpha_m} P \cdot D^{\alpha_{m+1}}
D^{\alpha_{m+2}}P \cdot S'$ where $S'$ is a string in $\mathcal
{A}(\alpha')$ and $\alpha': = (\alpha_{m+3}, \ldots, \alpha_n)$.
\end{longlist}

In all the above cases, $S'$ is an allowed string in the alphabet.

\begin{longlist}[(3)]
\item[(1)]
We use the inductive hypothesis to bound $I_{S'}^{\alpha'}(t_1,t,z_1)$
and the bound
%
%
\begin{equation}
\label{frstDerEst} \int_{\mathbb{R}^d} \bigl| D^{\alpha}P(t,z)\bigr| \,dz \leq
Ct^{-1/2}
\end{equation}
to get
\begin{eqnarray*}
&& \bigl| I_S^{\alpha}(t_0,t,z_0)\bigr|
\\
&&\qquad  =
\biggl\llvert\int_{t_0}^t \int
_{\mathbb{R}^d} b_1(t_1,z_1)
D^{\alpha_1}P(t_1 - t_0, z_1 -
z_0) I_{S'}^{\alpha
'}(t_1,t,z_1)
\,dz_1 \,dt_1 \biggr\rrvert
\\
&&\qquad  \leq \frac{M^{n-1}}{\Gamma( (n+1)/2)} \int_{t_0}^t
(t-t_1)^{(n-1)/2}\int_{\mathbb{R}^d}
\bigl|D^{\alpha_1}P(t_1 - t_0, z_1 -
z_0) \bigr| \,dz_1 \,dt_1
\\
&&\qquad  \leq \frac{M^{n-1}C}{\Gamma( (n+1)/2)} \int_{t_0}^t
(t-t_1)^{(n-1)/2}(t_1 - t_0)^{-1/2}
\,dt_1
\\
&&\qquad  = \frac{M^{n-1}C \sqrt{\pi} (t-t_0)^{k/2}}{\Gamma( (n/2) + 1)}.
\end{eqnarray*}
The result follows if $M \geq\max\{ C \sqrt{ \pi}, 1 \}$.

\item[(2)]
For this case, we can write
\begin{eqnarray*}
&& I_S^{\alpha} (t_0,t,
z_0)
\\
&&\qquad = \int_{t_0}^t \int
_{t_1}^t \int_{\mathbb{R}^d} \int
_{\mathbb{R}^d} b_1(t_1,z_1)
b_2(t_2,z_2)P(t_1-t_0, z_1 - z_0)
\\
&&\hspace*{94pt}{}\times
D^{\alpha_1}D^{\alpha_2} P(t_2 - t_1,
z_2 - z_1) I_{S'}^{\alpha'}(t_2,t,z_2)
\,dz_1\,dz_2 \,dt_2\,dt_1.
\end{eqnarray*}
We set $h(t_2,z_2): = b_2(t_2,z_2) I_{S'}^{\alpha'}(t_2,z_2)
(t-t_2)^{1-n/2}$ so that by the inductive hypothesis we have
\[
\| h \|_{\infty} \leq M^{n-2}/\Gamma(n/2).  %
\]
Use the above estimate in the first assertion of Lemma~\ref{cnvLmm}
with $g = b_1$ and integrate with respect to $t_2$ first, to get
\[
\bigl|I_S^{\alpha}(t_0,t,z_0)\bigr| \leq
\frac{CM^{n-2}(t-t_0)^{n/2}}{n\Gamma
(n/2)}
\]
and the result follows if $M \geq\max\{C, 1 \}$.

\item[(3)]
We have
\begin{eqnarray*}
\hspace*{-5pt}&& I_S^{\alpha}(t_0,t,z_0)
\\[-1pt]
\hspace*{-5pt}&&\!\!\qquad = \int_{t_0 < \cdots
t_{m+2} < t} \int_{\mathbb{R}^{(m+2)d}}
P(t_1 - t_0, z_1 - z_0)
\\[-2pt]
\hspace*{-5pt}&&\!\!\hspace*{119pt}{}\times  \prod
_{j=1}^{m+2} b_j(t_j,z_j)
\\[-1pt]
\hspace*{-5pt}&&\!\!\hspace*{119pt}{} \times\prod_{j=2}^{m}
D^{\alpha_j}P(t_j - t_{j-1},z_j -
z_{j-1})
\\[-2pt]
\hspace*{-5pt}&&\!\!\hspace*{145pt}{}\times D^{\alpha_{m+1}}D^{\alpha_{m+2}} P(t_{m+2} -
t_{m+1},z_{m+2} - z_{m+1})
\\[-1pt]
\hspace*{-5pt}&&\!\!\hspace*{145pt}{} \times I_{S'}^{\alpha'}(t_{m+2},t,z_{m+2})
\,dz_1 \cdots \,dz_{m+2} \,dt_1 \cdots
\,dt_{m+2}.
\end{eqnarray*}
Set $h(t_{m+2},z_{m+2}): = b_{m+2}(t_{m+2},z_{m+2}) I_{S'}^{\alpha
'}(t_{m+2},t,z)(t-t_{m+2})^{(2+m-n)/2}$. Then from the inductive
hypothesis we have $\| h\|_{\infty} \leq M^{n-m-2}/\Gamma((n-m)/2)$.
Define
\begin{eqnarray*}
\hspace*{-5pt}&& A(t_m,z_m)
\\[-1pt]
\hspace*{-5pt}&&\qquad: = \int_{t_m}^t
\int_{t_{m+1}}^t \int_{\mathbb{R}^{2d}}
b_{m+1}(t_{m+1},z_{m+1}) h(t_{m+2},z_{m+2})(t-t_{m+2})^{(n-m-2)/2}
\\[-2pt]
\hspace*{-5pt}&&\hspace*{92pt}{}\times D^{\alpha_m}P(t_{m+1}
- t_m, z_{m+1} - z)D^{\alpha_{m+1}}D^{\alpha_{m+2}}
\\[-1pt]
\hspace*{-8pt}&&\hspace*{92pt}{} \times P(t_{m+2} -
t_{m+1}, z_{m+2} - z_{m+1}) \,dz_{m+1}
\,dz_{m+2} \,dt_{m+1}\,dt_{m+2}.
\end{eqnarray*}
Then Lemma~\ref{cnvLmm} implies that
\[
\bigl|A(t_m,z_m)\bigr| \leq\frac{C(n-m)^{-1/2}
M^{n-m-2}(t-t_m)^{(n-m-1)/2}}{\Gamma( (n-m)/2 )}.  %
\]
Using this in
\begin{eqnarray*}
I_S^{\alpha}(t_0,t,z_0)
&=& \int_{t_0 < \cdots t_{m+2}
< t} \int_{\mathbb{R}^{(m+2)d}} P(t_1
- t_0, z_1 - z_0)
\\[-2pt]
&&\hspace*{86pt}{}\times \prod _{j=1}^m b_j(t_j,z_j)
\\[-1pt]
&&\hspace*{86pt}{} \times\prod_{j=1}^{m-1}D^{\alpha_j}P(t_j
- t_{j-1}, z_j- z_{j-1})
\\[-2pt]
&&\hspace*{115pt}{}\times  \Omega(t_m,z_m)
\,dz_1 \cdots \,dz_m \,dt_1 \cdots \,dt_m
\end{eqnarray*}
and using the bound (\ref{frstDerEst}) several times gives
\begin{eqnarray*}
\hspace*{-3pt}&& \bigl|I_S^{\alpha}(t_0,t,z_0)\bigr|
\\
\hspace*{-3pt}&&\qquad  \leq C^{m+1} (n-m)^{-1/2} \frac{M^{n-m-2}}{\Gamma((n-m)/2)}
\\
\hspace*{-3pt}&&\quad\qquad{} \times\int_{t_0 < \cdots t_m < t} (t_2 - t_1)^{-1/2}
\cdots(t_m - t_{m-1})^{-1/2}(t-t_m)^{(n-m-1)/2}
\,dt_1 \cdots \,dt_m
\\
\hspace*{-3pt}&&\qquad  = C^{m+1}(n-m)^{-1/2} \frac{M^{n-m-2}\pi^{(m-1)/2} \Gamma(
(n-m+1)/{2}) }{\Gamma( (n-m)/{2}) \Gamma( (n/2) + 1) }
(t-t_0)^{n/2}.
\end{eqnarray*}
We can choose $M$ so large that the result holds. This completes the
induction argument.\quad\qed
\end{longlist}\noqed
\end{pf}

%
\begin{remark} \label{Doleans-DadeBound}
Assume $\psi\in L^{\infty}([0,1] \times\Omega; \mathbb{R}^d)$
is adapted to the filtration generated by the Brownian motion. Then we can
bound the Doleans--Dade exponential $\mathcal{E}(\int_0^1 \psi(u) \,dB_u
)$ in $L^p(\Omega)$ by an increasing continuous function of $\| \psi\|
_{L^{\infty}([0,1] \times\Omega)}$.

To see this, notice that $M_t: = \mathcal{E}(\int_0^t \psi(u) \,dB_u )$
is the unique solution to the linear SDE
\[
dM_t = M(t) \psi(t) \,dB_t,\qquad M_0 =1.
\]
By It\^o's formula, we get
\begin{eqnarray*}
E\bigl[ M^p_t \bigr] &=& 1 + \frac{p(p-1)}{2}\int
_0^t E\bigl[ M^p_u \bigl|
\psi(u) \bigr|^2 \bigr] \,du %
\\
&\leq& 1 + \frac{p(p-1)}{2}\| \psi\|_{L^{\infty}([0,1] \times\Omega)} \int_0^t
E\bigl[ M^p_u \bigr] \,du;
\end{eqnarray*}
and
\begin{eqnarray*}
E\bigl[ M^p_t \bigr]
&\leq&\exp\biggl\{
\frac{tp(p-1)\| \psi\|
_{L^{\infty}([0,1] \times\Omega)}}{2} \biggr\},
\end{eqnarray*}
where we have used Gr\"onwall's lemma in the last inequality.
\end{remark}

We are now ready to complete the proof of Proposition~\ref{sobolevestimate}.

\begin{pf*}{Proof of Proposition \protect\ref{sobolevestimate}}
Let $t \in[0,1]$. Iterating the linearized equation~(\ref{linearizedeq}), we obtain
\[
\frac{\partial}{\partial x}X_{t}^{x}=\mathcal{I}_{d}+\sum
_{n=1}^{\infty
}\int_{0<s_{1}<\cdots s_{n}<t}b^{\prime}
\bigl(s_{1},X_{s_{1}}^{x}\bigr) \cdots
b^{\prime}\bigl(s_{n},X_{s_{n}}^{x}
\bigr)\,ds_{1}\cdots \,ds_{n},
\]
where, as before, $b'$ stands for the spatial Jacobian matrix of $b$.
Let $p\in{}[1,\infty)$ and choose $r,s\in{}[1,\infty)$ such
that $sp=2^{q}$ for some integer $q$ and $\frac{1}{r}+\frac{1}{s}=1$. Then
by Girsanov's theorem and H\"{o}lder's inequality
\begin{eqnarray*}
\hspace*{-5pt}&& E \biggl[ \biggl|\frac{\partial}{\partial x}X_{t}^{x} \biggr|^{p}
\biggr]
\\
\hspace*{-5pt}&&\qquad = E \Biggl[ \biggl|%
\mathcal{I}_{d}+\sum
_{n=1}^{\infty}\int_{0<s_{1}<\cdots s_{n}<t}b^{\prime
}(s_{1},x+B_{s_{1}})
\cdots b^{\prime}(s_{n},x+B_{s_{n}})\,ds_{1}
\cdots \,ds_{n} \biggr|^{p}
\\
\hspace*{-5pt}&&\hspace*{225pt}{}\times  \mathcal{E}\biggl(\int
_{0}^{1}b(u,x+B_{u})\,dB_{u}
\biggr) \Biggr]
\\
\hspace*{-5pt}&&\qquad \leq  C_{1}\bigl(\Vert b\Vert_{\infty}\bigr)
\Biggl\llVert
\mathcal{I}%
_{d}+\sum_{n=1}^{\infty}
\int_{0<s_{1}<\cdots s_{n}<t}b^{\prime
}(s_{1},x+B_{s_{1}})\cdots
\\
\hspace*{-5pt}&&\hspace*{177pt}{}  \times b^{\prime}(s_{n},x+B_{s_{n}})\,ds_{1}
\cdots \,ds_{n}\Biggr\rrVert_{L^{ps}(\Omega,\mathbb{R}^{d\times d})}^{p}\!,
\end{eqnarray*}
where $C_{1}$ is a continuous increasing function as in Remark~\ref{Doleans-DadeBound}.

Then we obtain
\begin{eqnarray*}
\hspace*{-5pt} && E \biggl|\frac{\partial}{\partial x} X_{t}^{x} \biggr|^{p}
\\
\hspace*{-5pt}&& \qquad \leq C_{1}\bigl(\Vert b\Vert_{\infty}\bigr)
\\
\hspace*{-5pt}&&\quad\qquad{} \times
\Biggl\llVert
\mathcal{I}%
_{d}+\sum_{n=1}^{\infty}
\int_{0<s_{1}<\cdots s_{n}<t}b^{\prime
}(s_{1},x+B_{s_{1}})\cdots
\\
\hspace*{-5pt}&&\hspace*{144pt}{} \times b^{\prime}(s_{n},x+B_{s_{n}})\,ds_{1}
\cdots \,ds_{n}\Biggr\rrVert_{L^{sp}(\Omega,\mathbb{R}^{d\times d})}^{p}
\\
\hspace*{-5pt}&&\qquad \leq C_{1}\bigl(\Vert b\Vert_{\infty}\bigr)
\\
\hspace*{-5pt}&&\quad\qquad{}\times  \Biggl( 1+\sum
_{n=1}^{\infty
}\sum_{i,j=1}^{d}
\sum_{l_{1}, \ldots, l_{n-1}=1}^{d}\biggl\llVert\int
_{0<s_{1}<\cdots<s_{n}<t}\frac{\partial}{\partial x_{l_{1}}}%
b^{(i)}(s_{1},x+B_{s_{1}})
\\
\hspace*{-5pt}&&\hspace*{218pt}{}\times
\frac{\partial}{\partial x_{l_{2}}}%
b^{(l_{1})}(s_{2},x+B_{s_{2}})\cdots
\\
\hspace*{-5pt}&&\hspace*{218pt}{} \times\frac{\partial}{\partial x_{j}}%
b^{(l_{n-1})}(s_{n},x+B_{s_{n}})
\\
\hspace*{-5pt}&&\hspace*{246pt}{}\times ds_{1}
\cdots \,ds_{n}\biggr\rrVert_{L^{ps}(\Omega,%
\mathbb{R})} \Biggr) ^{p}\!\!.
\end{eqnarray*}

Now consider the expression
\begin{eqnarray*}
A &:=& \int_{0< s_1 < \cdots< s_n < t} \frac{\partial}{\partial x_{l_1}}
b^{(i)}(s_1,
x+B_{s_1}) \frac{\partial}{\partial x_{l_2}} b^{(l_1)}(s_2,x+B_{s_2})
\cdots
\\
&&\hspace*{59pt}{}\times
\frac{\partial}{\partial x_{l_n}} b^{(l_n)}(s_n, x+B_{s_n})
\,ds_1 \cdots \,ds_n.
\end{eqnarray*}
Then, using (deterministic) integration by parts, repeatedly, it is
easy to
see that $A^2$ can be written as a sum of at most $2^{2n}$ terms of the form
%
%
\begin{equation}
\int_{0 < s_1 < \cdots< s_{2n} < t} g_1 (s_1) \cdots
g_{2n}(s_{2n}) \,ds_1 \cdots \,ds_{2n},
\end{equation}
where $g_l \in\{ \frac{\partial}{\partial x_j} b^{(i)}(\cdot, x+B_{
\cdot}) \dvtx 1 \leq i,j \leq d \} $, $l = 1, 2, \ldots, 2n$.
Similarly, by
induction it follows that $A^{2^q}$ is the sum of at most $2^{q2^qn}$ terms
of the form
%
%
\begin{equation}
\int_{0 < s_1 < \cdots< s_{2n} < t} g_1 (s_1) \cdots
g_{2^qn}(s_{2^qn}) \,ds_1 \cdots \,ds_{2^qn}.
\end{equation}

Combining this with Lemma~\ref{oldestimate}, 
we obtain the following estimate:
\begin{eqnarray*}
&& \biggl\llVert\int_{0<s_{1}<\cdots<s_{n}<t}\frac{\partial}{\partial
x_{l_{1}}}%
b^{(i)}(s_{1},x+B_{s_{1}})\frac{\partial}{\partial x_{l_{2}}}%
b^{(l_{1})}(s_{2},x+B_{s_{2}})\cdots
\\
&&\hspace*{104pt}{}\times
\frac{\partial}{\partial x_{j}}%
b^{(l_{n-1})}(s_{n},x+B_{s_{n}})\,ds_{1}\cdots \,ds_{n}\biggr\rrVert
_{L^{2^{q}}(\Omega, \mathbb{R})}
\\
&&\qquad \leq\biggl( \frac{2^{q2^{q}n}C^{2^{q}n}\Vert b\Vert_{\infty
}^{2^{q}n}t^{2^{q-1}n}}{\Gamma(2^{q-1}n+1)} \biggr) ^{2^{-q}}\leq\frac{
2^{qn}C^{n}\Vert b\Vert_{\infty}^{n}}{((2^{q-1}n)!)^{2^{-q}}}.
\end{eqnarray*}

Then it follows that
\begin{eqnarray*}
E \biggl[ \biggl| \frac{\partial}{\partial x} X_t^x \biggr|^p
\biggr]
&\leq& C_1\bigl(\|b\|_{\infty}\bigr) \Biggl(1 + \sum
_{n=1}^{\infty} \frac{%
d^{n+2}2^{qn}C^n \|b\|_{\infty}^n}{((2^{q-1}n)!)^{2^{-q}}} \Biggr)^p
= C_{d,p}\bigl(\|b\|_{\infty}\bigr).
\end{eqnarray*}
The right-hand side of this inequality is independent of $x \in\mathbb
{R}^d$%
, and the result follows.
\end{pf*}


For the rest of the paper, we will fix a bounded and measurable
$b\dvtx[0,1] \times\mathbb{R}^d \rightarrow\mathbb{R}^d$. It is proved in
\cite{Veretennikov} (and \cite{PMNPZ}) that the corresponding SDE (\ref
{FlowSDE}) has a unique strong solution, denoted by  $X_{.}^{s,x}$. Suppose
$b_n\dvtx[0,1] \times\mathbb{R}^d \rightarrow\mathbb{R}^d$ is a sequence
of compactly supported smooth functions such that
$b_{n}(t,x)\rightarrow b(t,x) \,dt\times dx$\mbox{-a.e.} and for some
positive constant $M$, $\llvert b_{n}(t,x)\rrvert\leq M<\infty$ for
all $n,t,x$. 
Denote by $X_{.}^{n,s,x}$ the solution of (\ref{FlowSDE}) when $b$ is
replaced by
$b_{n},n\geq1$.
We then have the following.

%
\begin{lemma}\label{WeakConvergenceOfSolutions}
Fix $s,t \in\mathbb{R}$ and $x \in\mathbb{R}^d$. Then the sequence
$X_{t}^{n,s,x}$ converges weakly in $L^2(\Omega; \mathbb{R}^d)$ to
$X_{t}^{s,x}$.
\end{lemma}

\begin{pf} 
For simplicity, consider $d =1$ and $s=0$. We start by noting that the set
\[
\biggl\{ \mathcal{E}\biggl(\int_0^1 h(u)
\,dB_u\biggr) \dvtx h \in C^1_b(\mathbb{R})
\biggr\} %
\]
spans a dense subspace of $L^2( \Omega; \mathbb{R})$. So, it suffices
to prove the convergence $E[ X_t^{n,x} \mathcal{E}(\int_0^1 h(u) \,dB_u)
] \rightarrow E[ X_t^{x} \mathcal{E}(\int_0^1 h(u) \,dB_u) ]$.

By the Cameron--Martin theorem, we have
\[
E\biggl[ X_t^{x} \mathcal{E}\biggl(\int
_0^1 h(u) \,dB_u\biggr) \biggr] =
\int_{ \Omega} X_t^x (\omega+ h) \,d \mu(
\omega). %
\]
The function $(u,x) \mapsto b(u,x) + h'(u)$ is still bounded, and so
$X_t^x( \cdot+ h)$ must coincide with the solution to (\ref{FlowSDE})
when $b$ is replaced by $b + h'$. Hence, by uniqueness in law of (\ref
{FlowSDE}), we may write
\[
E\biggl[ X_t^{x} \mathcal{E}\biggl(\int
_0^1 h(u) \,dB_u\biggr) \biggr] =
E\biggl[ (x + B_t) \mathcal{E}\biggl(\int_0^1
\bigl[b(u,x+B_u) + h'(u)\bigr] \,dB_u\biggr) \biggr]
\]
and similarly for $X_t^{n,x}$. We thus get
\begin{eqnarray*}
&& E\biggl[ X_t^{n,x} \mathcal{E}\biggl(\int
_0^1 h(u) \,dB_u\biggr) \biggr] -
E\biggl[ X_t^{x} \mathcal{E}\biggl(\int
_0^1 h(u) \,dB_u\biggr) \biggr]
\\
&&\qquad =E\biggl[(x+B_t) \biggl( \mathcal{E}\biggl(\int_0^1
\bigl[b_n(u,x+B_u) + h'(u)\bigr] \,dB_u
\biggr)
\\
&&\hspace*{103pt}{} - \mathcal{E}\biggl(\int_0^1
\bigl[b(u,x+B_u) + h'(u)\bigr] \,dB_u\biggr) \biggr)
\biggr].  %
\end{eqnarray*}

Using the inequality $| e^a - e^b | \leq|e^a + e^b | |a-b|$, H\"older's
inequality and Burkholder--Davis--Gundy inequality we find a constant
$C$ such that the above is bounded by
\begin{eqnarray*}
&& C \biggl( E\biggl[ \biggl(\mathcal{E}\biggl(\int_0^1
\bigl[b_n(u,x+B_u) + h'(u)\bigr] \,dB_u
\biggr)
\\
&&\hspace*{34pt}{} + \mathcal{E}\biggl(\int_0^1
\bigl[b(u,x+B_u) + h'(u)\bigr] \,dB_u\biggr)
\biggr)^4 \biggr] \biggr)^{1/4}
\\
&&\qquad {} \times\biggl( E\biggl[ \biggl( \int_0^1 \bigl(
b_n(u,x+B_u) - b(u,x+B_u)
\bigr)^2 \,du \biggr)^2
\\
&&\hspace*{55pt}{} + \biggl( \int_0^1
\bigl(b(u,x+B_u) + h'(u)\bigr)^2
\\
&&\hspace*{97pt} - \bigl(b_n(u,x+B_u) + h'(u)
\bigr)^2 \,du \biggr)^4\biggr] \biggr)^{1/4}.
\end{eqnarray*}
From Remark~\ref{Doleans-DadeBound}, since $b_n$ is uniformly bounded
we get that $ \{ \mathcal{E}(\int_0^1 [b_n(u,x+B_u) + h'(u)] \,dB_u) \}_n $
is bounded in $L^4(\Omega)$, so that the first factor above is
uniformly bounded. The second factor converges to zero by bounded convergence.
\end{pf}

We can actually strengthen the above lemma to get the following theorem.

%
\begin{theorem}\label{th13}
For any fixed $s,t \in\mathbb{R}$ and $x \in\mathbb{R}^d$, the
sequence $\{X_t^{n,s,x}\}_{n=1}^\infty$ converges strongly in
$L^2(\Omega; \mathbb{R}^d)$ to $X_t^{s,x}$.
\end{theorem}

\begin{pf}
For simplicity, consider the special case $s = 0$. We first give a
sketch of the proof that $\{X_t^{n,x}\}_{n=1}^\infty$
is relatively compact in $L^2(\Omega; \mathbb{R}^d)$. 
We notice that by Corollary~\ref{compactcrit}, it is enough to find a
constant $C >0$ such that
%
%
\begin{equation}
\label{compactnessEstimate1} \sup_{n}E\bigl[ \bigl| D_{\theta}
X_t^{n,x} - D_{\theta'} X_t^{n,x}
\bigr|^2\bigr] \leq C \bigl| \theta- \theta'\bigr|
\end{equation}
for $\theta, \theta' \in[0,t]$ and
%
%
\begin{equation}
\label{compactnessEstimate2} \sup_{n} \sup_{\theta\in[0,t]} E
\bigl[ \bigl| D_{\theta} X_t^{n,x} \bigr|^2\bigr]
\leq C.
\end{equation}
We begin by noticing that the Malliavin derivative satisfies the
linearized equation
\[
D_{\theta}X_t^{n,x} = \mathcal{I}_{d} +
\int_{ \theta}^t b'
\bigl[\bigl(u,X_u^{n,x}\bigr) D_{\theta}X_u^{n,x}\bigr]
\,du,  %
\]
which is the same equation as for $\frac{\partial}{\partial x}
X_t^{n,x}$ when we let $\theta=0$. The above inequalities can then be
obtained in a similar manner as for the bounds developed in Proposition
\ref{sobolevestimate}. Indeed, we may iterate the above linearized
equation to obtain
%
\[
D_{\theta}X_{t}^{n,x}=\mathcal{I}_{d}+
\sum_{k=1}^{\infty}\int_{ \theta
<u_{1}<\cdots u_{k}<t}b_n^{\prime}
\bigl(u_{1},X_{u_{1}}^{n,x}\bigr) \cdots
b^{\prime}_n\bigl(u_{k},X_{u_{k}}^{n,x}
\bigr)\,du_{1}\cdots \,du_{k}. %
\]
As in the proof of Proposition~\ref{sobolevestimate} with $p = 2$, we
get the bound
\[
E\bigl[\bigl| D_{\theta} X_t^{n,x} \bigr|^2\bigr]
\leq C_{d,2}\bigl(\|b\|_{\infty}\bigr),
\]
where the right-hand side is independent of $n,\theta, t$ and $x$. This
proves (\ref{compactnessEstimate2}).

Suppose now that $\theta< \theta'$, and write
\begin{eqnarray*}
&& D_{\theta} X_t^{n,x} - D_{\theta'}
X_t^{n,x}
\\
&&\qquad = \int_{\theta}^t
b'_n\bigl(u,X_u^{n,x}\bigr) D_{\theta}
X_u^{n,x} \,du - \int_{\theta'}^t
b'_n\bigl(u,X_u^{n,x}\bigr) D_{\theta'}
X_u^{n,x} \,du
\\
&&\qquad  = \int_{\theta}^{\theta'} b'_n
\bigl(u,X_u^{n,x}\bigr) D_{\theta}
X_u^{n,x} \,du
+ \int_{\theta'}^t b'_n
\bigl(u,X_u^{n,x}\bigr) \bigl(D_{\theta}
X_u^{n,x} - D_{\theta
'} X_u^{n,x}
\bigr) \,du
\\
&&\qquad  = D_{\theta} X_{\theta'}^{n,x} - \mathcal{I}_d
+ \int_{\theta'}^t b'_n
\bigl(u,X_u^{n,x}\bigr) \bigl(D_{\theta}
X_u^{n,x} - D_{\theta'} X_u^{n,x}
\bigr) \,du.
\end{eqnarray*}
%
Iterating the above linear equation, we get
\begin{eqnarray*}
&& D_{\theta}X_t^{n,x} - D_{\theta'}X_t^{n,x}
\\
&&\qquad = \Biggl( \mathcal{I}_d + \sum_{k=1}^{\infty}
\int_{\theta' < u_1 <
\cdots< u_k < t} b'_n\bigl(u_1,X_{u_1}^{n,x}
\bigr) \cdots b'_n\bigl(u_k,X_{u_k}^{n,x}
\bigr) \,du_1 \cdots \,du_k \Biggr)
\\
&&\quad\qquad{}\times  \bigl(D_{\theta}
X_{\theta'}^{n,x} - \mathcal{I}_d\bigr).  %
\end{eqnarray*}

On the other hand, note that
\[
D_{\theta} X_{\theta'}^{n,x} - \mathcal{I}_d
= \sum_{k=1}^{\infty} \int_{\theta< u_1 < \cdots< u_k < \theta'}
b'_n\bigl(u_1,X_{u_1}^{n,x}\bigr)
\cdots b'_n\bigl(u_k,X_{u_k}^{n,x}
\bigr) \,du_1 \cdots \,du_k %
\]
and so
\begin{eqnarray*}
&& E\bigl[\bigl|D_{\theta}X_t^{n,x} - D_{\theta'}X_t^{n,x}\bigr|^2
\bigr]
\\
&&\qquad \leq E \Biggl[ \Biggl\llvert\mathcal{I}_d + \sum
_{k=1}^{\infty} \int_{\theta
' < u_1 < \cdots< u_k < t}
b'_n(u_1,x + B_{u_1}) \cdots
\\
&&\hspace*{155pt}{}\times
b'_n(u_k,x + B_{u_k}) \,du_1 \cdots \,du_k \Biggr\rrvert^2
\\
&&\quad\qquad\hspace*{13pt}  {}\times\Biggl\llvert\sum_{k=1}^{\infty} \int
_{\theta< u_1 < \cdots< u_k <
\theta'} b'_n(u_1,x + B_{u_1})\cdots
\\
&&\hspace*{144pt}{}\times
b'_n(u_k,x + B_{u_k}) \,du_1
\cdots \,du_k \Biggr\rrvert^2
\\
&&\hspace*{162pt}{} \times\mathcal{E} \Biggl( \sum_{j=1}^d
\int_0^1 b^{(j)}_n(u,B_u)
\,dB_u \Biggr) \Biggr].
\end{eqnarray*}
By a similar argument as in the proof of Proposition~\ref{sobolevestimate}, we get
\[
E\bigl[\bigl|D_{\theta}X_t^{n,x} - D_{\theta'}X_t^{n,x}\bigr|^2
\bigr] \leq C_{d,2}\bigl(\| b \| _{\infty}\bigr) \bigl|\theta' -
\theta\bigr|,
\]
which proves (\ref{compactnessEstimate1}).

Let $\{X_t^{n_k,s,x} \}_{k \geq1}$ be a subsequence of $\{X_t^{n,s,x}
\}_{n \geq1}$. Applying the above compactness criterion to this
subsequence, we have that this subsequence is relatively compact in $L^2 (\Omega,{\mathbb R}^d)$. Thus,
we can extract a further subsequence which by Lemma~\ref{WeakConvergenceOfSolutions} must converge strongly to the limit
$X_t^{s,x}$. Since $L^2(\Omega; \mathbb{R}^d)$ is a Banach space, the
full sequence must converge strongly to $X_t^{s,x}$.
\end{pf}


As a consequence of Proposition~\ref{sobolevestimate} and the above
discussion, we obtain the following
result.

%
\begin{corollary}
\label{Holder} Let $X_{.}^{s,x}$ be the unique strong solution to the
SDE (%
\ref{FlowSDE}) and $q>1$ an integer. Then there exists a constant $%
C=C(d,\llVert b\rrVert_{\infty},q)<\infty$ such that
\[
E \bigl[ \bigl\llvert X_{t_{1}}^{s_{1},x_{1}}-X_{t_{2}}^{s_{2},x_{2}}
\bigr\rrvert^{q} \bigr] \leq C\bigl(\llvert s_{1}-s_{2}
\rrvert^{q/2}+\llvert t_{1}-t_{2}\rrvert
^{q/2}+\llvert x_{1}-x_{2}\rrvert^{q}
\bigr)
\]
for all $s_{1},s_{2},t_{1},t_{2},x_{1},x_{2}$.

In particular, there exists a continuous version of the random field\break $%
(s,t,x)\longmapsto X_{t}^{s,x}$ with H\"{o}lder continuous trajectories of
H\"older constant $\alpha<\frac{1}{2}$ in $s,t$ and $\alpha<1$ in $%
x $, locally (see \cite{Kunita}).
\end{corollary}

\begin{pf}
Retain the above notation. Without loss of generality, let $0\leq
s_{1}<s_{2}<t_{1}<t_{2}$. Then
\begin{eqnarray*}
&&X_{t_{1}}^{n,s_{1},x_{1}}-X_{t_{2}}^{n,s_{2},x_{2}}
\\
&& \qquad =x_{1}-x_{2}+\int_{s_{1}}^{t_{1}}b_{n}
\bigl(u,X_{u}^{n,s_{1},x_{1}}\bigr)\,du-%
\int
_{s_{2}}^{t_{2}}b_{n}\bigl(u,X_{u}^{n,s_{2},x_{2}}
\bigr)\,du
\\
&&\quad\qquad{}+(B_{t_{1}}-B_{s_{1}})-(B_{t_{2}}-B_{s_{2}})
\\
&&\qquad =x_{1}-x_{2}+\int_{s_{1}}^{s_{2}}b_{n}
\bigl(u,X_{u}^{n,s_{1},x_{1}}\bigr)\,du+%
\int
_{s_{2}}^{t_{1}}b_{n}\bigl(u,X_{u}^{n,s_{1},x_{1}}
\bigr)\,du
\\
&&\quad\qquad{}-\int_{s_{2}}^{t_{1}}b_{n}
\bigl(u,X_{u}^{n,s_{2},x_{2}}\bigr)\,du-%
\int
_{t_{1}}^{t_{2}}b_{n}\bigl(u,X_{u}^{n,s_{2},x_{2}}
\bigr)\,du
\\
&&\quad\qquad{} +(B_{t_{1}}-B_{t_{2}})+(B_{s_{2}}-B_{s_{1}})
\\
&&\qquad =x_{1}-x_{2}+\int_{s_{1}}^{s_{2}}b_{n}
\bigl(u,X_{u}^{n,s_{1},x_{1}}\bigr)\,du-%
\int
_{t_{1}}^{t_{2}}b_{n}\bigl(u,X_{u}^{n,s_{2},x_{2}}
\bigr)\,du
\\
&&\quad\qquad{} +\int_{s_{2}}^{t_{1}}
\bigl(b_{n}\bigl(u,X_{u}^{n,s_{1},x_{1}}
\bigr)-b_{n}\bigl(u,X_{u}^{n,s_{1},x_{2}}\bigr)\bigr)\,du
\\
&&\quad\qquad{} +\int_{s_{2}}^{t_{1}}
\bigl(b_{n}\bigl(u,X_{u}^{n,s_{1},x_{2}}
\bigr)-b_{n}\bigl(u,X_{u}^{n,s_{2},x_{2}}\bigr)\bigr)\,du
\\
&&\quad\qquad{} +(B_{t_{1}}-B_{t_{2}})+(B_{s_{2}}-B_{s_{1}}).
\end{eqnarray*}
So, due to the uniform boundedness of $b_{n},n\geq1$, we get
%
%
\begin{eqnarray} \label{Eq1}
&&E\bigl[\bigl\llvert X_{t_{1}}^{n,s_{1},x_{1}}-X_{t_{2}}^{n,s_{2},x_{2}}
\bigr\rrvert^{q}\bigr]
\nonumber
\\
&&\qquad \leq C_{q}\biggl(\llvert x_{1}-x_{2}\rrvert
^{q}+\llvert s_{1}-s_{2}\rrvert
^{q/2}+\llvert t_{1}-t_{2}\rrvert
^{q/2}
\nonumber\\[-8pt]\\[-8pt]
&&\hspace*{50pt}{}+ E\biggl[\biggl\llvert\int_{s_{2}}^{t_{1}}
\bigl(b_{n}\bigl(u,X_{u}^{n,s_{1},x_{1}}
\bigr)-b_{n}\bigl(u,X_{u}^{n,s_{1},x_{2}}\bigr)\bigr)\,du
\biggr\rrvert^{q}\biggr]
\nonumber
\\
&&\hspace*{50pt}{}+E\biggl[\biggl\llvert\int_{s_{2}}^{t_{1}}
\bigl(b_{n}\bigl(u,X_{u}^{n,s_{1},x_{2}}
\bigr)-b_{n}\bigl(u,X_{u}^{n,s_{2},x_{2}}\bigr)\bigr)\,du
\biggr\rrvert^{q}\biggr]\biggr).\nonumber
\end{eqnarray}
Using the fact that $X_{t}^{n,\cdot,s}$ is a stochastic flow of
diffeomorphisms (see, e.g., \cite{Kunita}), the mean value theorem and
Proposition %
\ref{sobolevestimate}, we get
%
%
\begin{eqnarray} \label{Eq2}
&&E\biggl[\biggl\llvert\int_{s_{2}}^{t_{1}}
\bigl(b_{n}\bigl(u,X_{u}^{n,s_{1},x_{1}}
\bigr)-b_{n}\bigl(u,X_{u}^{n,s_{1},x_{2}}\bigr)\bigr)\,du
\biggr\rrvert^{q}\biggr]\nonumber
\\
&&\qquad =\llvert x_{1}-x_{2}\rrvert^{q}\nonumber
\\
&&\quad\qquad{}\times E\biggl[\biggl\llvert\int_{s_{2}}^{t_{1}}\int
_{0}^{1}\biggl(b_{n}^{\shortmid }\bigl(u,X_{u}^{n,s_{1},x_{1}+\tau(x_{2}-x_{1})}\bigr)\frac{\partial
}{\partial x}X_{u}^{n,s_{1},x_{1}+\tau(x_{2}-x_{1})}\biggr)\,d\tau \,du\biggr\rrvert^{q}
\biggr]\nonumber
\\
&&\qquad \leq\llvert x_{1}-x_{2}\rrvert^{q}\nonumber
\\
&&\quad\qquad{}\times \int
_{0}^{1}E\biggl[\biggl\llvert\int
_{s_{2}}^{t_{1}}\biggl(b_{n}^{\shortmid}
\bigl(u,X_{u}^{n,s_{1},x_{1}+\tau
(x_{2}-x_{1})}\bigr)\frac{\partial}{\partial x}X_{u}^{n,s_{1},x_{1}+\tau
(x_{2}-x_{1})}
\biggr)\,du\biggr\rrvert^{q}\biggr]\,d\tau
\\
&&\qquad =\llvert x_{1}-x_{2}\rrvert^{q}\nonumber
\\
&&\quad\qquad{}\times \int_{0}^{1}E\biggl[\biggl\llvert
\frac{\partial}{\partial x}X_{t_{1}}^{n,s_{1},x_{1}+\tau(x_{2}-x_{1})}-\frac
{\partial}{\partial x}X_{s_{2}}^{n,s_{1},x_{1}+\tau
(x_{2}-x_{1})}
\biggr\rrvert^{q}\biggr]\,d\tau\nonumber
\\
&&\qquad \leq C_{q}\llvert x_{1}-x_{2}\rrvert
^{q}\sup_{t\in{}[
s_{1},1],x\in\mathbb{R}^{d}}E\biggl[\biggl\llvert
\frac{\partial}{\partial x}%
X_{t_{1}}^{n,s_{1},x}\biggr\rrvert
^{q}\biggr]
\nonumber
\\
&&\qquad \leq C_{d,q}\bigl(\llVert b\rrVert_{\infty}\bigr)\llvert
x_{1}-x_{2}\rrvert^{q}.\nonumber
\end{eqnarray}
Finally, we observe that estimation of the last term of the right-hand side
of (\ref{Eq1}) can be reduced to the previous case (\ref{Eq2}) by applying
the Markov property, since
\begin{eqnarray*}
&& E\biggl[\biggl\llvert\int_{s_{2}}^{t_{1}}
\bigl(b_{n}\bigl(u,X_{u}^{n,s_{1},x_{2}}
\bigr)-b_{n}\bigl(u,X_{u}^{n,s_{2},x_{2}}\bigr)\bigr)\,du
\biggr\rrvert^{q}\biggr]
\\
&&\qquad \leq \int_{s_{2}}^{t_{1}}E\bigl[\bigl\llvert
b_{n}\bigl(u,X_{u}^{n,s_{1},x_{2}}\bigr)-b_{n}
\bigl(u,X_{u}^{n,s_{2},x_{2}}\bigr)\bigr\rrvert^{q}
\bigr]\,du
\\
&&\qquad =\int_{s_{2}}^{t_{1}}E\bigl[ E\bigl[\bigl
\llvert b_{n}\bigl(u,X_{u}^{n,s_{2},y}
\bigr)-b_{n}\bigl(u,X_{u}^{n,s_{2},x_{2}}\bigr)\bigr
\rrvert^{q}\bigr]\rrvert_{y=X_{s_{2}}^{n,s_{1},x_{2}}}\bigr]\,du
\\
&&\qquad \leq CE\bigl[\bigl\llvert X_{s_{2}}^{n,s_{1},x_{2}}-x_{2}
\bigr\rrvert^{q}\bigr]=CE\bigl[\bigl\llvert
X_{s_{2}}^{n,s_{1},x_{2}}-X_{s_{1}}^{n,s_{1},x_{2}}
\bigr\rrvert^{q}\bigr]
\\
&&\qquad \leq M_{q}\llvert s_{2}-s_{1}\rrvert
^{q/2}
\end{eqnarray*}
for a positive constant $M_{q}<\infty$.

Therefore, we have
\[
E\bigl[\bigl\llvert X_{t_{1}}^{n,s_{1},x_{1}}-X_{t_{2}}^{n,s_{2},x_{2}}
\bigr\rrvert^{q}\bigr]\leq C_{q}\bigl(\llvert
s_{1}-s_{2}\rrvert^{q/2}+\llvert
t_{1}-t_{2}\rrvert^{q/2}+\llvert
x_{1}-x_{2}\rrvert^{q}\bigr)
\]
for a constant $C_{q}$ independent of $n$.

To complete the proof of the corollary, we use the fact that
$X_{t_{1}}^{n,s_{1},x_{1}}%
\rightarrow X_{t_{1}}^{s_{1},x_{1}}$ and $X_{t_{2}}^{n,s_{2},x_{2}}%
\rightarrow X_{t_{2}}^{s_{2},x_{2}}$ strongly in $L^{2}(\Omega; \mathbb
{R}^d )$ as $n\rightarrow
\infty$ (Theorem~\ref{th13}),
together with Fatou's lemma applied to a.e.
convergent subsequences of $\{X_{t_{1}}^{n,s_{1},x_{1}}\}_{n=1}^{\infty}$
and $\{X_{t_{2}}^{n,s_{2},x_{2}}\}_{n=1}^{\infty}$.
\end{pf}

This concludes step one of our program. We next proceed to step~2.

For simplicity, we consider $s=0$, that is, we look at the sequence\break $
\{ X_t^{n,x} \}_{n \geq1}: = \{ X_t^{n,0,x} \}_{n \geq1}$ and $X_t^x:= X_t^{0,x} $.
The following lemma establishes convergence of the above sequence.

%
\begin{lemma}
\label{compactness} For any $\varphi\in C_{0}^{\infty}(\mathbb{R}^{d};
\mathbb{R}^{d})$ and $t\in{}[0,1]$, the sequence
\[
\bigl\langle X_{t}^{n},\varphi\bigr\rangle=\int
_{\mathbb{R}^{d}}\bigl\langle X_{t}^{n,x},\varphi(x)
\bigr\rangle_{\mathbb{R}^{d}}\,dx
\]
converges to $\langle X_{t},\varphi\rangle$ in $L^{2}(\Omega,\mathbb
{R} )$.
\end{lemma}

\begin{pf}
Denote by $D_{s}$ the Malliavin derivative (see the \hyperref[sec4]{Appendix}) and by
$U$ the
compact support of $\varphi$. By noting the inequalities
\[
E\bigl[\bigl|D_{s}\bigl\langle X_{t}^{n},\varphi
\bigr\rangle\bigr|^{2}\bigr]=E\bigl[\bigl|\bigl\langle D_{s}X_{t}^{n},
\varphi\bigr\rangle\bigr|^{2}\bigr]
\leq\Vert\varphi\Vert_{L^{2}(\mathbb{R}^{d})}^{2}|U|\sup_{x\in
U}E
\bigl[\bigl|D_{s}X_{t}^{n,x}\bigr|^{2}\bigr]
\]
and
\begin{eqnarray*}
&& E\bigl[\bigl|D_{s}\bigl\langle X_{t}^{n},\varphi
\bigr\rangle_{L^{2}(\mathbb{R}%
^{d})}-D_{s^{\prime}}\bigl\langle X_{t}^{n},
\varphi\bigr\rangle\bigr|^{2}\bigr]
\\
&&\qquad =E\bigl[\bigl|\bigl\langle D_{s}X_{t}^{n}-D_{s^{\prime}}X_{t}^{n},
\varphi\bigr\rangle\bigr|^{2}\bigr]
\\
&&\qquad \leq\Vert\varphi\Vert_{L^{2}(\mathbb{R}^{d})}^{2}|U|\sup_{x\in
U}E
\bigl[\bigl|D_{s}X_{t}^{n,x}-D_{s^{\prime}}X_{t}^{n,x}\bigr|^{2}
\bigr]
\end{eqnarray*}
we can invoke Corollary~\ref{compactcrit} in the \hyperref[sec4]{Appendix}
to obtain a subsequence $\langle X_{t}^{n(k)},\varphi\rangle
$ converging in $L^{2}(\Omega,\mathbb{R})$ as $%
k \rightarrow\infty$. Denote the limit by $Y(\varphi)$.

Similar to the proof of Lemma~\ref{WeakConvergenceOfSolutions} one can
show that $E[ \langle X_{t}^{n},\varphi\rangle\mathcal{E}\times \break(\int_0^1
h(u) \,dB_u) ]$ converges to $E[ \langle X_{t},\varphi\rangle\mathcal
{E}(\int_0^1 h(u) \,dB_u) ]$ for all $h \in C_b^1(\mathbb{R}; \mathbb
{R}^d)$. We then get that $\langle X_{t}^{n},\varphi\rangle$ converges
weakly to $\langle X_{t},\varphi\rangle$, and so, by uniqueness of the
limits, we
can conclude that
\[
Y(\varphi)=\langle X_{t},\varphi\rangle.
\]
To see that the full sequence converges, we assume that there exist an
$\varepsilon> 0$ and a subsequence $\langle X_{t}^{n(k)},\varphi\rangle
$ such that
\[
\bigl\| \bigl\langle X_{t}^{n(k)},\varphi\bigr\rangle- \langle
X_{t},\varphi\rangle\bigr\| \geq\varepsilon%
\]
for every $k$. Applying the above procedure to $\langle
X_{t}^{n(k)},\varphi\rangle$ gives a further subsequence converging
to $\langle X_{t},\varphi\rangle$ thus giving a contradiction.
\end{pf}

We are now able to finalize the proof of Proposition~\ref{MainProposition}.

\begin{pf*}{Proof of Proposition \protect\ref{MainProposition}}
Using Proposition~\ref{sobolevestimate}, we have
\[
\sup_{n} \sup_{x \in\mathbb{R}^d}E \biggl[ \biggl|
\frac{\partial
}{\partial x}X_{t}^{n,x} \biggr|^{p} \biggr] <
\infty.
\]
Hence, there exists a subsequence of $\frac{\partial}{\partial x}%
X_{t}^{n(k),x}$ converging in the weak
topology of $L^{2}(\Omega,L^{p}(U))$ to an element $Y$. Then we have for
any $A\in\mathcal{F}$ and $\varphi\in C_{0}^{\infty}(U; \mathbb{R}^d)$
\begin{eqnarray*}
E\bigl[1_{A}\bigl\langle X_{t},\varphi^{\prime}
\bigr\rangle\bigr]&=&\lim_{k\rightarrow
\infty
}E\bigl[1_{A}\bigl
\langle X_{t}^{n(k)},\varphi^{\prime}\bigr\rangle\bigr]
\\
&=& -\lim_{k\rightarrow\infty}E\biggl[1_{A}\biggl\langle
\frac{\partial}{\partial
x}%
X_{t}^{n(k)},\varphi\biggr
\rangle\biggr]=-E\bigl[1_{A}\langle Y,\varphi\rangle\bigr].
\end{eqnarray*}
Hence, we have for $\varphi\in C_{0}^{\infty}$:
%
%
\begin{equation}
\label{weakDerivative} \bigl\langle X_{t},\varphi^{\prime}\bigr
\rangle=-\langle Y,\varphi\rangle
\end{equation}
$P$-a.s. Finally, we need to show that there exists a measurable set
$\Omega_0 \subset\Omega$ with full measure such that $X_t^{\cdot} $
has a weak derivative on this subset. To this end, choose a sequence
$\{ \varphi_n \}$ in $C^{\infty}(U;\mathbb{R}^d)$ dense in
$W^{1,2}_0(U;\mathbb{R}^d)$. Choose a measurable subset $\Omega_n$ of
$\Omega$ with full measure such that (\ref{weakDerivative}) holds on
$\Omega_n$ with $\varphi$ replaced by $\varphi_n$. Then $\Omega_0: =
\bigcap_{n \geq1}\Omega_n$ satisfies the desired property.
\end{pf*}

We now return to the weighted Sobolev spaces. Using the same techniques as
in the above lemma, we prove the following.

%
\begin{lemma}
\label{weightedsobolev} For all $p \in(1, \infty)$, we have
\[
X_t^{\cdot} \in L^2\bigl(\Omega, W^{1,p}
\bigl(\mathbb{R}^d, w\bigr)\bigr).
\]
\end{lemma}

\begin{pf}
For simplicity, we consider the case $d=1$. It suffices to show that
$E[(\int| \frac{\partial}{%
\partial x} X_t^x |^p w(x) \,dx)^{2/p}] < \infty$. To this end, let $X_t^{n,x}$
denote the sequence approximating $X_t^x$ as in the previous lemma. Assume
first that $p \geq2$. Then by H\"{o}lder's inequality w.r.t. the
Wiener measure $\mu$, we have
\begin{eqnarray*}
&& E\biggl[\biggl(\int\biggl|\frac{\partial}{\partial x} X_t^{n,x}
\biggr|^p w(x) \,dx \biggr)^{2/p}\biggr]
\\
&&\qquad \leq\biggl(E\int\biggl|\frac{\partial}{\partial x} X_t^{n,x}
\biggr|^p w(x) \,dx \biggr)^{2/p}
\\
&&\qquad \leq\biggl(\int w(x) \,dx
\biggr)^{p/2} \biggl(\sup_{x \in\mathbb{R}} E \biggl| \frac{\partial
}{\partial x}
X_t^{n,x} \biggr|^p \biggr)^{2/p}.
\end{eqnarray*}
For $1 < p \leq2$, by H\"{o}lder's inequality w.r.t. $w(x)\,dx$, we have
\[
E\biggl[\biggl(\int\biggl|\frac{\partial}{\partial x} X_t^{n,x}
\biggr|^p w(x) \,dx \biggr)^{2/p}\biggr] \leq\biggl(\int w(x)\,dx
\biggr)^{(4-p)/2} \sup_{x \in\mathbb{R}^d} E\biggl[\biggl| \frac{\partial}{%
\partial x}
X_t^{n,x} \biggr|^2\biggr].
\]
In both cases, we can find a subsequence of $\frac{\partial}{\partial x} X^{n,x}_t$ converging to an element $Y \in
L^2(\Omega, L^p(\mathbb{R}^d,w))$ in the weak topology. In particular for
every $A \in\mathcal{F}$ and $f \in L^q(\mathbb{R}^d,w)$ ($q$ is the
Sobolev conjugate of $p$) we have
\[
\lim_{k \rightarrow\infty} E\biggl[ 1_A \int\frac{\partial}{\partial x}
X_t^{n(k),x} f(x) w(x) \,dx\biggr] = E\biggl[ 1_A
\int Y(x) f(x) w(x) \,dx\biggr]
\]
by choosing $f$ such that $fw \in L^q(\mathbb{R}, dx)$ [e.g., put $f(x) =
e^{-w(x)}\varphi(x)$ for $\varphi\in C_0^{\infty}(\mathbb{R})$]. It
follows that
$Y$ must coincide with the weak derivative of $%
X_t^x$. This proves the lemma.
\end{pf}

We now complete the proof of our main theorem in this section (Theorem~\ref{MainTheorem}) and its
corollary.

\begin{pf*}{Proof of Theorem \protect\ref{MainTheorem}}
Denote by $\mathbb{R}\times\mathbb{R}\times\mathbb{R}^{d}\ni
(s,t,x)\longmapsto\phi_{s,t}(x)\in\mathbb{R}^{d}$ the continuous version
of the solution map $(s,t,x)\longmapsto X_{t}^{s,x}$ provided by
Corollary \ref%
{Holder}. Let $\Omega^{\ast}$ be the set of all $\omega\in\Omega$
such that the SDE (\ref{FlowSDE}) has a unique spatially Sobolev
differentiable family
of solutions. Then by completeness of the probability space $(\Omega,%
\mathcal{F},\mu)$, it follows that $\Omega^{\ast}\in\mathcal{F}$
and $%
\mu(\Omega^{\ast})=1$. Furthermore, by uniqueness of solutions of
the SDE
(\ref{FlowSDE}), it is easy to check that the following two-parameter group
property
%
%
\begin{equation}
\phi_{s,t}(\cdot,\omega)=\phi_{u,t}(\cdot,\omega)\circ\phi
_{s,u}(\cdot,\omega),\qquad\phi_{s,s}(x,\omega)=x,
\label{Group}
\end{equation}
holds for all $s,u,t\in\mathbb{R}$, all $x\in\mathbb{R}^{d}$ and all $
\omega\in\Omega^{\ast}$. Finally, we apply Lemma~\ref{weightedsobolev}
and use the relation $\phi_{s,t}(\cdot,\omega)=\phi_{t,s}^{-1}(\cdot,\omega)$, to complete the proof of the theorem.
\end{pf*}

\begin{pf*}{Proof of Corollary \ref{CorollaryAut}}
Let $\Omega^{\ast}$ denote the set of full Wiener measure introduced in
the above proof of Theorem~\ref{MainTheorem}. We claim that $\theta
(t,\cdot
)(\Omega^{\ast})=\Omega^{\ast}$ for all $t\in\mathbb{R}$. To see this,
let $\omega\in\Omega^{\ast}$ and fix an arbitrary $t_{1}\in\mathbb{R}$.
Then from the autonomous SDE (\ref{SDEAut}) it follows that
%
%
\begin{equation}
\qquad X_{t+t_{1}}^{t_{1},x}(\omega)=x+\int_{t_{1}}^{t+t_{1}}b
\bigl(X_{u}^{t_{1},x}(\omega)\bigr)\,du+B_{t+t_{1}}(\omega
)-B_{t_{1}}(\omega),\qquad t_{1},t\in\mathbb{R}. \label{SDECor1}
\end{equation}
By the helix property of $B$ and a simple change of variable the above
relation implies
%
%
\begin{equation}
X_{t+t_{1}}^{t_{1},x}(\omega)=x+\int_{0}^{t}b
\bigl(X_{u+t_{1}}^{t_{1},x}(\omega)\bigr)\,du+B_{t}\bigl(
\theta\bigl(t_{1}(\omega)\bigr)\bigr),\qquad t\in\mathbb{R}. \label{SDECor2}
\end{equation}
The above relation implies that the SDE (\ref{SDEAut}) admits a Sobolev
differentiable family of solutions when $\omega$ is replaced by $\theta
(t_{1},\omega)$. Hence, $\theta(t_{1},\omega)\in\Omega^{\ast}$.
Thus $%
\theta(t_{1},\cdot)(\Omega^{\ast})\subseteq\Omega^{\ast}$, and
since $%
t_{1}\in\mathbb{R}$ is arbitrary, this proves our claim. Furthermore, using
uniqueness in the integral equation (\ref{SDECor1}) it follows that
%
%
\begin{equation}
X_{t_{2}+t_{1}}^{t_{1},x}(\omega)=X_{t_{2}}^{0,x}
\bigl(\theta(t_{1},\omega)\bigr) \label{Theta}
\end{equation}
for all $t_{1},t_{2}\in\mathbb{R}$, all $x\in\mathbb{R}^{d}$ and
$\omega
\in\Omega^{\ast}$. To prove the following cocycle property for all $%
\omega\in\Omega^{\ast}$:
\[
\phi_{0,t_{1}+t_{2}}(\cdot,\omega)=\phi_{0,t_{2}}\bigl(\cdot,\theta
(t_{1},\omega)\bigr)\circ\phi_{0,t_{1}}(\cdot,\omega)
\]
we rewrite the identity (\ref{Theta}) in the form
%
%
\begin{equation}
\phi_{t_{1},t_{1}+t_{2}}(x,\omega)=\phi_{0,t_{2}}\bigl(x,\theta
(t_{1},\omega)\bigr),\qquad t_{1},t_{2}\in\mathbb{R},
x\in\mathbb{R}^{d},\omega\in\Omega^{\ast},
\end{equation}
replace $x$ by $\phi_{0,t_{1}}(x,\omega)$ in the above
identity and invoke the two-parameter flow property (\ref{Group}). This
completes the proof of Corollary~\ref{CorollaryAut}.
\end{pf*}

Finally, we give an extension of Theorem~\ref{MainTheorem} to a class of
nondegenerate \mbox{$d$-}dimensional It\^{o}-diffusions.

%
\begin{theorem}
\label{generalsde}Consider the time-homogeneous $\mathbb{R}^{d}$-valued SDE
%
%
\begin{equation}
dX_{t}^{x}=b\bigl(X_{t}^{x}\bigr)\,dt+
\sigma\bigl(X_{t}^{x}\bigr)\,dB_{t},\qquad
X_{0}^{x}=x\in\mathbb{R}^{d}, 0\leq t
\leq1, \label{SDEDiffusion}
\end{equation}
where the coefficients $b\dvtx\mathbb{R}^{d}\rightarrow\mathbb{R}^{d}$
and $%
\sigma\dvtx\mathbb{R}^{d}\rightarrow\mathbb{R}^{d}\times \mathbb
{R}^{d}$%
are Borel measurable. Suppose that $\sigma(x)$ has an inverse $\sigma
^{-1}(x)$ for all $x\in\mathbb{R}^{d}$. Further assume that $\sigma
^{-1}\dvtx%
\mathbb{R}^{d}\rightarrow\mathbb{R}^{d}\times \mathbb{R}^{d}$ is
continuously differentiable such that
\[
\frac{\partial}{\partial x_{k}}\sigma_{lj}^{-1}=\frac{\partial
}{\partial
x_{j}}
\sigma_{lk}^{-1}
\]
for all $l,k,j=1,\ldots,d$. In addition, require that the function
$\Lambda\dvtx%
\mathbb{R}^{d}\rightarrow\mathbb{R}^{d}$ defined by
\[
\Lambda(x):=\int_{0}^{1}\sigma^{-1}(tx)
\cdot x\,dt
\]
possesses a Lipschitz continuous inverse $\Lambda^{-1}\dvtx\mathbb{R}%
^{d}\rightarrow\mathbb{R}^{d}$. Let $D\Lambda\dvtx\mathbb{R}%
^{d}\rightarrow\break  L ( \mathbb{R}^{d},\mathbb{R}^{d} ) $ and $
D^{2}\Lambda\dvtx\mathbb{R}^{d}\rightarrow L ( \mathbb{R}^{d}\times
\mathbb{R}^{d},\mathbb{R}^{d} ) $ be the existing corresponding
derivatives of~$\Lambda$.

Assume that the function $b^{\ast}\dvtx\mathbb{R}^{d}\rightarrow
\mathbb{R}%
^{d}$ given by
\begin{eqnarray*}
b^{\ast}(x) &:=& D\Lambda\bigl( \Lambda^{-1} ( x ) \bigr)
\bigl[ b\bigl(\Lambda^{-1} ( x ) \bigr) \bigr]
\\
&&{}+\frac{1}{2}D^{2}\Lambda\bigl( \Lambda^{-1} ( x )
\bigr) \Biggl[ \sum_{i=1}^{d}\sigma\bigl(
\Lambda^{-1} ( x ) \bigr) [ e_{i} ], \sum
_{i=1}^{d}\sigma\bigl(\Lambda^{-1} ( x )
\bigr) [ e_{i} ] %
\Biggr]
\end{eqnarray*}
is bounded and Borel measurable, where $e_{i}$, $i=1,\ldots,d$, is a
basis of $%
\mathbb{R}^{d}$.

Then there exists a stochastic flow $(s,t,x)\longmapsto\phi_{s,t}(x)$ of
the SDE (\ref{SDEDiffusion}) such that
\[
\phi_{s,t}(\cdot)\in L^{2}\bigl(\Omega,W^{p}
\bigl(\mathbb{R}^{d},w\bigr)\bigr)
\]
for all $0\leq s\leq t\leq1$ and all $p>1$.
\end{theorem}

\begin{pf}
Because of our assumptions, we see that $\Lambda^{-1}$ is twice continuously
differentiable and that
\[
D\Lambda(y)\sigma(y)=\mathcal{I}_{d}
\]
for all $y\in\mathbb{R}^{d}$.

Then It\^{o}'s lemma applied to (\ref{FlowSDE}) implies that
\begin{eqnarray*}
dY_{t}^{x} &=&D\Lambda\bigl( \Lambda^{-1}
\bigl( Y_{t}^{x} \bigr) \bigr) \bigl[ b\bigl(
\Lambda^{-1} \bigl( Y_{t}^{x} \bigr) \bigr) \bigr]
\\
&&{} +\frac{1}{2}D^{2}\Lambda\bigl(
\Lambda^{-1} \bigl( Y_{t}^{x} \bigr) \bigr)
\\
&&\quad{} \times
\Biggl[ \sum_{i=1}^{d}\sigma\bigl(
\Lambda^{-1} \bigl( Y_{t}^{x} \bigr) \bigr) [
e_{i} ], \sum_{i=1}^{d}\sigma
\bigl(\Lambda^{-1} \bigl( Y_{t}^{x} \bigr) \bigr)
[ e_{i} ] \Biggr] \,dt+dB_{t},
\\
Y_{0}^{x} &=&\Lambda( x ), \qquad0\leq t\leq1,
\end{eqnarray*}
where $Y_{t}^{x}=\Lambda( X_{t}^{x} ) $. Because of Theorem
\ref{MainTheorem} and a chain rule for functions in Sobolev spaces (see, e.g.,
\cite{Ziemer}) there exists a stochastic flow $(s,t,x)\longmapsto\phi
_{s,t}(x)$ of the SDE (\ref{SDEDiffusion}) such that $\phi_{s,t}(\cdot
)\in
L^{2}(\Omega,W^{p}(\mathbb{R}^{d},w))$ for all $0\leq s\leq t\leq1$ and
all $p>1$.
\end{pf}

\section{Application to the stochastic transport equation}\label{sec3}

In this section, we will study the stochastic transport equation
%
%
\begin{equation}
\cases{ \displaystyle d_{t}u(t,x)+ \bigl(b(t,x)\cdot Du(t,x)\bigr)\,dt+\sum
_{i=1}^{d}e_{i}\cdot Du(t,x)\circ
dB_{t}^{i}=0,
\cr
u(0,x)= u_{0}(x),}\label{TransportEqDiff}
\end{equation}
where $e_{1}, \ldots, e_{d}$ is the canonical basis of $\mathbb
{R}^{d}$, $%
b\dvtx[0,1]\times\mathbb{R}^{d}\rightarrow\mathbb{R}^{d}$ is a given
bounded measurable vector
field and $u_{0}\dvtx\mathbb{R}^{d}\rightarrow\mathbb{R}$ is a given initial
data. The stochastic integration is understood in the Stratonovich sense.

In \cite{Kunita}, it is proved that for smooth data and a sufficiently regular
vector field~$b$, (\ref{TransportEqDiff}) has an explicit solution
$u(t,x) =
u_0(\phi_t^{-1}(x))$ where $\phi_t(x)$ is the flow map generated by the
strong solutions $(X_t^x)_{t \geq0}$ of the SDE (\ref{FlowSDE}). In
fact, this
solution of the transport equation is strong in the sense that
$u(t,\cdot)$ is differentiable everywhere in $x$ almost
surely for all $t$, and it satisfies the integral equation
\[
u(t,x) + \int_0^t Du(s,x) \cdot b(s,x) \,ds +
\sum_{i=1}^d \int_0^t
e_i \cdot Du(s,x) \circ dB^i_s =
u_0(x)
\]
almost surely, for every $t$.

We shall use the following notion of weak solution
(cf. Definition 12 in \cite{FlandoliGubinelliPriola}).

%
\begin{definition}
Let $b$ be bounded and measurable and $u_{0}\in L^{\infty}(\mathbb
{R}^{d})$. A~\textit{weak solution}
of the transport equation (%
\ref{TransportEqDiff}) is a stochastic process $u\in L^{\infty}(\Omega
\times{}[0,1]\times\mathbb{R}^{d})$ such that, for every $t$, the
function $u(t,\cdot)$ is weakly differentiable a.s. with $\sup_{0\leq
s\leq
1,x\in\mathbb{R}^{d}}E[\llvert Du(s,x)\rrvert^{4}]<\infty$
and for\vspace*{1pt}
every test function $\theta\in C_{0}^{\infty}(\mathbb{R}^{d})$, the
process $\int_{\mathbb{R}^{d}}\theta(x)u(t,x)\,dx$ has a continuous
modification which is an $\mathcal{F}_{t}$-semi-martingale satisfying
%
%
\begin{eqnarray}\label{WeakSolution}
\int_{\mathbb{R}^{d}}\theta(x)u(t,x)\,dx& =&\int_{\mathbb{R}^{d}}
\theta(x)u_{0}(x)\,dx
\nonumber
\\
&&{} -\int_{0}^{t}\int_{\mathbb{R}^{d}}Du(s,x)
\cdot b(s,x)\theta(x)\,dx\,ds
\\
&&{} +\sum_{i=1}^{d}\int
_{0}^{t} \biggl( \int_{\mathbb
{R}^{d}}u(s,x)D_{i}
\theta(x)\,dx \biggr) \circ dB_{s}^{i}, \nonumber
\end{eqnarray}
where $Du(t,x)$ is the weak derivative of $u(t,x)$ in the following space-variable.
\end{definition}

Our definition of a weak 
solution for (\ref{TransportEqDiff}) differs slightly from that in \cite
{FlandoliGubinelliPriola} due to the fact that we do not require any regularity
on the coefficient $b$ except Borel measurability and boundedness. To
compensate for it, the expression depends on
the weak derivative of $u(t,x)$.

It is easy to see that equation (\ref{WeakSolution}) can be written in
the equivalent It\^o form.

%
\begin{lemma}
\label{EquivalentDef} A process $u\in L^{\infty}(\Omega\times{}[
0,1]\times\mathbb{R}^{d})$ is a weak 
solution
of the transport equation (\ref{TransportEqDiff}) if and only if, for
every $t$, the
function $u(t,\cdot)$ is weakly differentiable a.s. with $\sup_{0\leq
s\leq
1,x\in\mathbb{R}^{d}}E[\llvert Du(s,x)\rrvert^{4}]<\infty$,
and for every
test function $\theta\in C_{0}^{\infty}(\mathbb{R}^{d})$, the process
$%
\int_{\mathbb{R}^{d}}\theta(x)u(t,x)\,dx$ has a continuous $\mathcal
{F}_{t}$%
-adapted modification satisfying the following equation a.s.:
\begin{eqnarray*}
\int_{\mathbb{R}^{d}}\theta(x)u(t,x)\,dx& =&\int_{\mathbb{R}^{d}}
\theta(x)u_{0}(x)\,dx
\\
&&{} -\int_{0}^{t}\int_{\mathbb{R}^{d}}Du(s,x)
\cdot b(s,x)\theta(x)\,dx\,ds
\\
&&{} +\sum_{i=1}^{d}\int
_{0}^{t} \biggl( \int_{\mathbb
{R}^{d}}u(s,x)D_{i}
\theta(x)\,dx \biggr) \,dB_{s}^{i}
\\
&&{} +\frac{1}{2}\int_{0}^{t}\int
_{\mathbb{R}^{d}}u(s,x)\Delta\theta(x)\,dx\,ds.
\end{eqnarray*}
\end{lemma}

The main result of this section is the following existence and uniqueness
theorem for solutions of the stochastic transport equation (\ref
{TransportEqDiff}).

%
\begin{theorem}\label{ExistenceOfWeakSolution} Let $b$ be bounded and Borel
measurable. Suppose
$u_{0}\in C_{b}^{1}(\mathbb{R}^{d})$. Then there exists a unique weak
solution $u(t,x)$ to the stochastic transport equation (\ref%
{TransportEqDiff}). For each $t > 0$ and all $p \in(1, \infty)$, the
weak solution $u(t, \cdot)$ belongs a.s. to $W^{1,p}(\mathbb
{R}^{d},w)$, the weighted Sobolev space introduced in Section~\ref{sec1}.
Moreover, for fixed $t$ and $x$, $u(t,\cdot,x)$ is Malliavin-differentiable.
\end{theorem}

%
\begin{remark}
As noted in \cite{FlandoliGubinelliPriola},
the deterministic transport equation is generally ill-posed under the
conditions of Theorem~\ref{ExistenceOfWeakSolution}. It is remarkable that
Brownian forcing on the transport equation induces
uniqueness and regularity of the solution.
\end{remark}

We shall prove Theorem~\ref{ExistenceOfWeakSolution} using a sequence
$b_{n}\dvtx[0,1]\times\mathbb{R}%
^{d}\rightarrow\mathbb{R}^{d}$ of uniformly bounded smooth
functions with compact support converging almost everywhere to~$b$. We then
study the corresponding sequence of solutions of the transport equation
(\ref{TransportEqDiff})
when $b$ is replaced by $b_{n}$.

For the rest of this section, we denote by $\phi_t$ the flow of the SDE
(\ref{FlowSDE})
driven by the vector field $b$, and by $\phi_{n,t}$ the flow of the SDE
(\ref{FlowSDE})
with $b_n$ in place of~$b$.

We begin with the following lemma.

%
\begin{lemma}
\label{TransformConvergence} Let $u_0 \in C^1_b(\mathbb{R}^d)$ and $f
\in
L^1(\mathbb{R}^d)$. Then the sequence
\[
\biggl( \int_{\mathbb{R}^d} u_0\bigl(
\phi_{n,s}^{-1}(x)\bigr)f(x) \,dx \biggr)_{n
\geq1}
\]
converges to $\int_{\mathbb{R}^d} u_0(\phi_s^{-1}(x))f(x) \,dx $ in $%
L^2(\Omega)$ for every $s \in[0,1]$.
\end{lemma}

\begin{pf}
Consider
\begin{eqnarray*}
&& \biggl\| \int_{\mathbb{R}^d} u_0\bigl(\phi_{n,s}^{-1}(x)
\bigr)f(x) \,dx - \int_{\mathbb{R}^d} u_0\bigl(
\phi_s^{-1}(x)\bigr)f(x) \,dx \biggr\|_{L^2(\Omega)}
\\
&&\qquad \leq\int_{\mathbb{R}^d} \bigl\|u_0\bigl(
\phi_{n,s}^{-1}(x)\bigr) - u_0\bigl(
\phi_s^{-1}(x)\bigr)\bigr\|_{L^2(\Omega)} \bigl|f(x)\bigr| \,dx.
\end{eqnarray*}
We have $\|u_0(\phi_{n,s}^{-1}(x)) - u_0(\phi_s^{-1}(x))\|_{L^2(\Omega)}
\leq\|Du_0 \|_{\infty} \|\phi_{n,s}^{-1}(x) -
\phi_s^{-1}(x)\|_{L^2(\Omega)} $ which goes to zero for every $s$ and
$x$. Now
\[
\bigl\|u_0\bigl(\phi_{n,s}^{-1}\bigr) -
u_0\bigl(\phi_s^{-1}\bigr)
\bigr\|_{L^2(\Omega)} |f| \leq2\|u_0\|_{\infty} |f| \in
L^1\bigl(\mathbb{R}^d\bigr)
\]
and the result follows by dominated convergence.
\end{pf}

We also need the following result (see Theorem 2 in \cite{Haj} and also
\cite{Re1,Re2}).

%
\begin{theorem}
\label{ChangeVariable}Let $\mathcal{U}$ be open subset of $\mathbb
{R}^{d}$ and
$f\in W^{1,d}(\mathcal{U})$ be a homeomorphism.
Then $f$ satisfies the Lusin's condition, that is,
\[
E\subset\mathcal{U},\qquad\llvert E\rrvert=0\quad\Longrightarrow\quad \bigl
\llvert f(E)
\bigr\rrvert=0.
\]
Here, $\llvert A\rrvert$ stands for the Lebesgue measure of a
set $A$.%

Moreover, for every measurable function $g\dvtx\mathcal
{U}\longrightarrow
{}[0,\infty)$ and a measurable set $E\subset\mathcal{U}$ the following
change of variable formula is valid:
\[
\int_{E}(g\circ f)\llvert\det Jf\rrvert \,dx=\int
_{f(E)}g(y)\,dy,
\]
where $\det Jf$ is the determinant of the Jacobian of $f$.
\end{theorem}

%
\begin{remark}
The random diffeomorphisms $\phi_{t}(\cdot),\phi_{t}^{-1}(\cdot)\in
W_{\mathrm{loc}}^{1,p}(\mathbb{R}^{d})$ a.s. and satisfy the conditions of Theorem
\ref{ChangeVariable} on each bounded and open subset $\mathcal{U}$ of
$\mathbb{R}%
^{d}$.
\end{remark}

We are now ready to prove Theorem~\ref{ExistenceOfWeakSolution}:

\begin{pf*}{Proof of Theorem \protect\ref{ExistenceOfWeakSolution}}
(1) \textit{Existence of a weak solution}.
We consider the approximation $\{b_{n}\}$ of $b$ as described prior to Lemma~\ref{WeakConvergenceOfSolutions}.
Then we know that there exists a unique strong solution to the
transport\vspace*{1pt} equation (\ref%
{TransportEqDiff}) when $b$ is replaced by~$b_{n}$, which is given
by $u_{n}(t,x)=u_{0}(\phi_{n,t}^{-1}(x)), n \geq1$. In particular,
$u_{n}$~is a
differentiable, weak $L^{\infty}$-solution,
such that for every $\theta\in C^{\infty}(\mathbb{R}%
^{d})$
\begin{eqnarray}
\label{equality} \int_{\mathbb{R}^{d}}\theta(x)u_{n}(t,x)\,dx&
=&\int_{\mathbb
{R}^{d}}\theta(x)u_{0}(x)\,dx
\nonumber
\\
&&{} -\int_{0}^{t}\int_{\mathbb{R}^{d}}Du_{n}(s,x)
\cdot b_{n}(s,x)\theta(x)\,dx\,ds
\nonumber\\[-8pt]\\[-8pt]
&&{} +\sum_{i=1}^{d}\int
_{0}^{t} \biggl( \int_{\mathbb{R}^{d}}u_{n}(s,x)D_{i}
\theta(x)\,dx \biggr) \,dB_{s}^{i}
\nonumber
\\
&&{} +\frac{1}{2}\int_{0}^{t}\int
_{\mathbb{R}^{d}}u_{n}(s,x)\Delta\theta(x)\,dx\,ds.
\nonumber
\end{eqnarray}
Let us now define $u(t,x):=u_{0}(\phi_{t}^{-1}(x))$ so that $u\in
L^{\infty
}(\Omega\times{}[0,1]\times\mathbb{R}^{d})$, and $u(t,\cdot)$ is
weakly differentiable, a.s. We now let $n$ go to infinity to get that $%
u(t,x)$ is a solution of the transport equation.

The following two limits exist in $L^2(\Omega)$ by Lemma~\ref{TransformConvergence} and dominated convergence:
%
\begin{eqnarray*}
\int_{\mathbb{R}^d} \theta(x) u_n(t,x) \,dx &\rightarrow&
\int_{\mathbb{R}^d} \theta(x) u(t,x) \,dx,
\\
\int_0^t \int_{\mathbb{R}^d}
u_n(s,x) \Delta\theta(x) \,dx \,ds &\rightarrow&\int_0^t
\int_{\mathbb{R}^d} u(s,x) \Delta\theta(x) \,dx \,ds.
\end{eqnarray*}
%
By the It\^{o} isometry, we have
\[
\sum_{i=1}^d \int_0^t
\biggl( \int_{\mathbb{R}^d} u_n(s,x)D_i
\theta(x) \,dx \biggr) \,dB^i_s \rightarrow\sum
_{i=1}^d \int_0^t
\biggl( \int_{\mathbb{R}^d} u(s,x)D_i \theta(x) \,dx \biggr)
\,dB^i_s
\]
in $L^2(\Omega)$. Finally, we claim that
\begin{eqnarray*}
&& \int_0^t \int_{\mathbb{R}^d}
Du_n(s,x)\cdot b_n(s,x) \theta(x) \,dx \,ds
\\
&&\qquad \rightarrow\int
_0^t \int_{\mathbb{R}^d} Du(s,x)
\cdot b(s,x) \theta(x) \,dx \,ds
\end{eqnarray*}
in $L^2(\Omega)$. To see this, observe that
\[
\biggl( \int_0^t \int_{\mathbb{R}^d}
Du_n(s,x)\cdot b_n(s,x) \theta(x) \,dx \,ds
\biggr)_n
\]
is convergent in $L^2(\Omega)$ because of the convergence of the other
terms in equality (\ref{equality}). Then the claim
is proved once we show that $\int_0^t \int_{\mathbb{R}^d} Du_n(s,x)\cdot
b_n(s,x) \theta(x) \,dx \,ds$ converges \textit{weakly} to $\int_0^t \int_{%
\mathbb{R}^d} Du(s,x)\cdot b(s,x) \theta(x) \,dx \,ds$. Then the strong and weak
limit must coincide.

To prove weak convergence, we write the difference in three parts,
namely:
\begin{eqnarray*}
&& \int_{0}^{t}\int_{\mathbb{R}^{d}}Du_{n}(s,x)
\cdot b_{n}(s,x)\theta(x)\,dx\,ds-\int_{0}^{t}
\int_{\mathbb{R}^{d}}Du(s,x)\cdot b(s,x)\theta(x)\,dx\,ds
\\
&&\qquad
=\int_{0}^{t}\int_{\mathbb{R}^{d}}Du_{n}(s,x)
\cdot b_{n}(s,x)\theta(x)\,dx\,ds
\\
&&\quad\qquad{} -\int_{0}^{t}
\int_{\mathbb{R}^{d}}Du_{n}(s,x)\cdot b(s,x)\theta(x)\,dx\,ds
\\
&&\qquad\quad {} +\int_{0}^{t}\int_{\mathbb{R}^{d}}Du_{0}
\bigl(\phi_{n,s}^{-1}(x)\bigr)D\phi_{n,s}^{-1}(x)
\cdot b(s,x)\theta(x)\,dx\,ds
\\
&&\quad\qquad{} -\int_{0}^{t}\int
_{\mathbb{R}%
^{d}}Du_{0}\bigl(\phi_{s}^{-1}(x)
\bigr)D\phi_{n,s}^{-1}(x)\cdot b(s,x)\theta(x)\,dx\,ds
\\
&&\quad\qquad{} +\int_{0}^{t}\int_{\mathbb{R}^{d}}Du_{0}
\bigl(\phi_{s}^{-1}(x)\bigr)D\phi_{n,s}^{-1}(x)
\cdot b(s,x)\theta(x)\,dx\,ds
\\
&&\quad\qquad{} -\int_{0}^{t}\int
_{\mathbb{R}%
^{d}}Du_{0}\bigl(\phi_{s}^{-1}(x)
\bigr)D\phi_{s}^{-1}(x)\cdot b(s,x)\theta(x)\,dx\,ds
\\
&&\qquad
=(i)_{n}+(\mathit{ii})_{n}+(\mathit{iii})_{n}.
\end{eqnarray*}
We shall deal with these terms separately.

($\alpha$): the first term $(i)_{n}$ converges to 0 strongly in
$L^{2}(\Omega)$ as $n \to\infty$,
since by H\"{o}lder's inequality and Fubini's theorem
\begin{eqnarray*}
E\bigl[(i)_{n}^{2}\bigr]&=&E \biggl[ \biggl( \int
_{0}^{t}\int_{\mathbb{R}%
^{d}}Du_{n}(s,x)
\cdot\bigl(b_{n}(s,x)-b(s,x)\bigr)\theta(x)\,dx\,ds \biggr)
^{2} \biggr]
\\
&\leq&\int_{0}^{t}\int_{\mathbb{R}%
^{d}}E
\bigl[\bigl|Du_{n}(s,x)\bigr|^{2}\bigr]\bigl|b_{n}(s,x)-b(s,x)\bigr|^{2}\bigl|
\theta(x)\bigr|\,dx\Vert\theta\Vert_{L^{1}(\mathbb{R})}.
\end{eqnarray*}
We have that
\[
E\bigl[\bigl|Du_{n}(s,x)\bigr|^{2}\bigr]\leq\Vert
Du_{0}\Vert_{\infty}^{2}E\bigl[\bigl|D\phi
_{n,s}^{-1}(x)\bigr|^{2}\bigr],
\]
which is uniformly bounded in $n$, $s$ and $x$ by Proposition~\ref{sobolevestimate}. Then, using dominated convergence, we obtain
$\lim_{n \to\infty} (i)_{n} =0$.

($\beta$): the second term converges strongly to 0 in $L^{2}(\Omega
)$, because of the following estimates:
\begin{eqnarray*}
E\bigl[(\mathit{ii})_{n}^{2}\bigr]&\leq&\Vert b\Vert_{\infty}^{2}
\\
&&\!{} \times E
\biggl( \int_{0}^{t}\int_{%
\mathbb{R}^{d}}\bigl|Du_{0}
\bigl(\phi_{n,s}^{-1}(x)\bigr)-Du_{0}\bigl(\phi
_{s}^{-1}(x)\bigr)\bigr|\bigl|D\phi_{n,s}^{-1}(x)\bigr|\bigl|
\theta(x)\bigr|\,dx\,ds \biggr) ^{2}
\\
&\leq&\Vert b\Vert_{\infty}^{2}t\Vert\theta\Vert_{L^{1}(\mathbb{R}%
^{d})}
\\
&&\!{}\times
\int_{0}^{t}\int_{\mathbb{R}^{d}}E
\bigl[\bigl|Du_{0}\bigl(\phi_{n,s}^{-1}(x)
\bigr)-Du_{0}\bigl(\phi_{s}^{-1}(x)
\bigr)\bigr|^{2}\bigl|D\phi_{n,s}^{-1}(x)\bigr|^{2}
\bigr]\bigl|\theta(x)\bigr|\,dx\,ds
\\
&\leq&\Vert b\Vert_{\infty}^{2}t\Vert\theta\Vert_{L^{1}(\mathbb{R}%
^{d})}
\\
&&\!{}\times \int_{0}^{t}\int_{\mathbb{R}^{d}} \bigl(
E\bigl[\bigl|Du_{0}\bigl(\phi_{n,s}^{-1}(x)
\bigr)-Du_{0}\bigl(\phi_{s}^{-1}(x)
\bigr)\bigr|^{4}\bigr] \bigr) ^{1/2}
\\
&&\hspace*{40pt}{}\times \bigl( E\bigl[\bigl|D\phi
_{n,s}^{-1}(x)\bigr|^{4}\bigr] \bigr)
^{1/2}\bigl|\theta(x)\bigr|\,dx\,ds
\\
&\leq&\Vert b\Vert_{\infty}^{2}t\Vert\theta\Vert_{L^{1}(\mathbb{R}%
^{d})}
\sup_{k,r,y} \bigl( E\bigl[\bigl|D\phi_{k,r}^{-1}(y)\bigr|^{4}
\bigr] \bigr) ^{1/2}
\\
&&\!{} \times\int_{0}^{t}\int_{\mathbb{R}^{d}}
\bigl( E\bigl[\bigl|Du_{0}\bigl(\phi_{n,s}^{-1}(x)
\bigr)-Du_{0}\bigl(\phi_{s}^{-1}(x)
\bigr)\bigr|^{4}\bigr] \bigr) ^{1/2}\bigl|\theta(x)\bigr|\,dx\,ds.
\end{eqnarray*}
The above estimates are consequences of H\"{o}lder's inequality. Since
$Du_{0}$ is
bounded and continuous, the right-hand side of the above inequality
converges to 0 by dominated convergence.

($\gamma$): for the last term, let $X\in L^{2}(\Omega)$ and consider
\begin{eqnarray*}
&& E\bigl[(\mathit{iii})_{n}X\bigr]
\\
&&\qquad =\int_{0}^{t}E
\biggl[\int_{\mathbb{R}^{d}}Du_{0}\bigl(\phi
_{s}^{-1}(x)\bigr) \bigl(D\phi_{n,s}^{-1}(x)-D
\phi_{s}^{-1}(x)\bigr)\cdot b(s,x)\theta(x)X\,dx\biggr]\,ds.
\end{eqnarray*}
Now, for each $s$, since $Du_{0}$, $b$ and $\theta$ are bounded and
$D\phi
_{s}^{-1}$ is the weak limit of~$D\phi_{n,s}^{-1}$, this expression tends
to 0 as $n\rightarrow\infty$.

(2) \textit{Uniqueness of weak solutions}.
Let us assume that $u$ is a weak
solution to the stochastic transport equation (\ref{WeakSolution})
(with $%
\sup_{0\leq s\leq1,x\in\mathbb{R}^{d}}E[\llvert Du(s,x)\rrvert
^{4}]<\infty)$. We will show that
\[
u(t,x)=u_{0}\bigl(\phi_{t}^{-1}(x)\bigr)\qquad\mbox{a.e.}
\]
This will guarantee uniqueness of the weak solution to the transport equation.
So, let $V$ be a bounded and open subset of $\mathbb{R}^{d}$ and
consider for the locally integrable function $u(t,\cdot)$ on $\mathbb
{R}%
^{d} $ its mollification
\[
u_{\varepsilon}(t,x)=(u\ast\eta_{\varepsilon})(x)=\int_{\mathbb
{R}^{d}}u(t,y)
\eta_{\varepsilon}(x-y)\,dy
\]
with respect to the standard mollifier $\eta$.

We observe that $u_{\varepsilon}$ satisfies the equation
\[
u_{\varepsilon}(t,x)=u_{0,\varepsilon}(x)-\int_{0}^{t}(b
\cdot Du)_{\varepsilon
}(s,x)\,ds-\int_{0}^{t}(Du)_{\varepsilon}(s,x)
\circ dB_{s}.
\]

Then using the It\^{o}--Ventzell formula applied to $u_{\varepsilon}$ and
$\phi
_{t}(x)$ (see \cite{Kunita}) gives
%
%
\begin{eqnarray}\label{1*}
&& u_{\varepsilon}\bigl(t,\phi_{t}(x)\bigr)
\nonumber\\[-9pt]\\[-12pt]
&&\qquad =u_{0,\varepsilon}(x)
+\int _{0}^{t}\bigl((Du)_{\varepsilon
}\bigl(s,
\phi_{s}(x)\bigr)\cdot b\bigl(s,\phi_{s}(x)\bigr)-(b\cdot
Du)_{\varepsilon}\bigl(s,\phi_{s}(x)\bigr)\bigr)\,ds.\hspace*{-17pt} \nonumber
\end{eqnarray}

Now let $\tau\in L^{\infty}(\Omega)$ and $\theta$ be a smooth function
with compact support in \mbox{$V \subseteq \Bbb{R}^d$}. Denote by $\chi_V$ the indicator function of $V$. Then it follows from (\ref{1*}) that
%
%
\begin{eqnarray}\label{2*}
&&E\biggl[\tau\int_{V}\theta(x)u_{\varepsilon}\bigl(t,
\phi_{t}(x)\bigr)\,dx\biggr]
\nonumber
\\[-1pt]
&&\qquad =E\biggl[\tau\int_{V}\theta(x)u_{0,\varepsilon}(x)\,dx
\biggr]
\\[-1pt]
&&\quad\qquad{}+E\biggl[\tau\int_{0}^{t}\int
_{V}\theta(x) \bigl((Du)_{\varepsilon}\bigl(s,\phi
_{s}(x)\bigr)\cdot b\bigl(s,\phi_{s}(x)\bigr)
\nonumber\\[-9pt]\\[-9pt]
&&\hspace*{120pt}{} -(b\cdot
Du)_{\varepsilon}\bigl(s,\phi_{s}(x)\bigr)\bigr)\,dx\,ds\biggr].\nonumber
\end{eqnarray}
Using Theorem~\ref{ChangeVariable} applied to $\phi_{t}^{-1}(\cdot)$, we
obtain
%
%
\begin{eqnarray}\label{3*}
&&E\biggl[\tau\int_{0}^{t}\int
_{V}\theta(x) \bigl((Du)_{\varepsilon}\bigl(s,\phi
_{s}(x)\bigr)\cdot b\bigl(s,\phi_{s}(x)\bigr)-(b\cdot
Du)_{\varepsilon}\bigl(s,\phi_{s}(x)\bigr)\bigr)\,dx\,ds\biggr]
\nonumber
\\[-1pt]
&&\qquad =E\biggl[\tau\int_{0}^{t}\int
_{\mathbb{R}^{d}}\chi_{\phi_{s}(V)}(x)\theta\bigl(
\phi_{s}^{-1}(x)\bigr)\nonumber
\\[-1pt]
&&\hspace*{84pt}{}\times  \bigl((Du)_{\varepsilon}(s,x)\cdot
b(s,x)-(b\cdot Du)_{\varepsilon
}(s,x)\bigr)
\\[-1pt]
&&\hspace*{165pt}{}\times
\bigl\llvert\det\bigl(J
\phi_{s}^{-1}(x)\bigr)\bigr\rrvert \,dx\,ds\biggr]
\nonumber
\\[-2pt]
&&\qquad =I_{1}+I_{2},\nonumber
\end{eqnarray}
where
%
%
\begin{eqnarray} \label{4*}
I_{1}&:=&E\biggl[\tau\int_{0}^{t}\int
_{\mathbb{R}^{d}}\chi_{\phi
_{s}(V)}(x)\theta\bigl(
\phi_{s}^{-1}(x)\bigr) \bigl((Du)_{\varepsilon}(s,x)\cdot
b(s,x)\bigr)
\nonumber\\[-9pt]\\[-9pt]
&&\hspace*{133pt}{}\times \bigl\llvert\det\bigl(J\phi_{s}^{-1}(x)
\bigr)\bigr\rrvert \,dx\,ds\biggr]\nonumber
\end{eqnarray}
and
%
%
\begin{eqnarray}\label{5*}
I_{2}&:=&-E\biggl[\tau\int_{0}^{t}\int
_{\mathbb{R}^{d}}\chi_{\phi
_{s}(V)}(x)\theta\bigl(\phi_{s}^{-1}(x)
\bigr) (b\cdot Du)_{\varepsilon}(s,x)
\nonumber\\[-9pt]\\[-9pt]
&&\hspace*{109.5pt}{}\times \bigl\llvert\det\bigl(J
\phi_{s}^{-1}(x)\bigr)\bigr\rrvert \,dx\,ds\biggr].\nonumber
\end{eqnarray}
Since $V$ is bounded, there exists a $n\in\mathbb{N}$ such that
$V\subset
\overline{V}\subset W:=(-n,n)^{d}$. Then we get
%
%
\begin{eqnarray}\label{6*}
\bigl\llVert(Du)_{\varepsilon}\bigr\rrVert_{L^{2}(\phi_{s}(V))} &\leq&
\llVert Du
\rrVert_{L^{2}(\phi_{s}(W))},
\nonumber
\\
\bigl\llVert(b\cdot Du)_{\varepsilon}\bigr\rrVert
_{L^{2}(\phi_{s}(V))} &\leq&\llVert b\cdot Du\rrVert_{L^{2}(\phi_{s}(W))}
\\
&\leq&\llVert b\rrVert_{\infty}\llVert Du\rrVert_{L^{2}(\phi_{s}(W))}.\nonumber
\end{eqnarray}

Using (\ref{6*}), H\"{o}lder's inequality,
Fubini's theorem and Theorem~\ref{ChangeVariable}, we obtain
%
%
\begin{eqnarray}\label{7*}
\qquad I_{1} &\leq&CE\biggl[\int_{0}^{t}
\biggl(\int_{\mathbb{R}^{d}}\bigl(\chi_{\phi
_{s}(V)}(x)\theta\bigl(
\phi_{s}^{-1}(x)\bigr)b(s,x)\bigl\llvert\det\bigl(J\phi
_{s}^{-1}(x)\bigr)\bigr\rrvert\bigr)^{2}\,dx
\biggr)^{1/2}
\nonumber
\\
&&\hspace*{113pt}{}\times\biggl(\int_{\mathbb{R}^{d}}\chi_{\phi_{s}(W)}(x)\bigl\llvert
Du(s,x)\bigr\rrvert^{2}\,dx\biggr)^{1/2}\,ds\biggr]
\nonumber
\\
&\leq&C\int_{0}^{t}E\biggl[\int
_{\mathbb{R}^{d}}\bigl(\chi_{\phi
_{s}(V)}(x)\theta\bigl(
\phi_{s}^{-1}(x)\bigr)b(s,x)\bigl\llvert\det\bigl(J
\phi_{s}^{-1}(x)\bigr)\bigr\rrvert\bigr)^{2}\,dx
\biggr]^{1/2}
\nonumber
\\
&&{}\times E\biggl[\int_{\mathbb{R}^{d}}\chi_{\phi_{s}(W)}(x)\bigl
\llvert Du(s,x)\bigr\rrvert^{2}\,dx\biggr]^{1/2}\,ds
\nonumber
\\
&\leq&C\int_{0}^{t}E\biggl[\int
_{\mathbb{R}^{d}}\chi_{\phi_{s}(V)}(x)\bigl\llvert\det\bigl(J
\phi_{s}^{-1}(x)\bigr)\bigr\rrvert^{2}\,dx
\biggr]^{1/2}
\nonumber
\\
&&{}\times E\biggl[\int_{\mathbb{R}^{d}}\chi_{\phi_{s}(W)}(x)\bigl
\llvert Du(s,x)\bigr\rrvert^{2}\,dx\biggr]^{1/2}\,ds
\\
&\leq&C\int_{0}^{t}\biggl(\int
_{\mathbb{R}^{d}}E\bigl[\chi_{\phi
_{s}(V)}(x)\bigr]^{{1}/{2}}E
\bigl[\bigl\llvert\det\bigl(J\phi_{s}^{-1}(x)\bigr)\bigr
\rrvert^{4}\bigr]^{1/{2}%
}\,dx\biggr)^{1/2}
\nonumber
\\
&&{}\times\biggl(\int_{\mathbb{R}^{d}}E\bigl[\chi_{\phi_{s}(W)}(x)
\bigr]^{1/2%
}E\bigl[\bigl\llvert Du(s,x)\bigr\rrvert^{4}
\bigr]^{1/2}\,dx\biggr)^{1/2}\,ds
\nonumber
\\
&\leq&C\sup_{0\leq s\leq1,x\in\mathbb{R}^{d}}E\bigl[\bigl\llvert\det
\bigl(J\phi
_{s}^{-1}(x)\bigr)\bigr\rrvert^{4}
\bigr]^{1/2}\sup_{0\leq s\leq1,x\in
\mathbb{%
R}^{d}}E\bigl[\bigl\llvert Du(s,x)
\bigr\rrvert^{4}\bigr]^{1/2}
\nonumber
\\
&&{}\times\int_{0}^{t}\biggl(\int
_{\mathbb{R}^{d}}E\bigl[\chi_{\phi
_{s}(V)}(x)\bigr]^{{1}/{2}}\,dx
\biggr)\,ds
\nonumber
\\
&\leq&C\int_{0}^{t}\biggl(\int
_{\mathbb{R}^{d}}E\bigl[\chi_{\phi
_{s}(V)}(x)\bigr]^{{1}/{2}}\,dx
\biggr)\,ds \nonumber
\end{eqnarray}
for a constant $C$ depending on the sizes of $V$, $\theta$ and $b$, since
\[
\sup_{0\leq s\leq1,x\in\mathbb{R}^{d}}E\bigl[\bigl\llvert\det\bigl
(J\phi
_{s}^{-1}(x)\bigr)\bigr\rrvert^{4}\bigr]\leq
M<\infty
\]
because of Proposition~\ref{sobolevestimate} applied to $%
\phi_{s}^{-1}(x)$.

Further, it follows from Girsanov's theorem, H\"{o}lder's inequality
and the
symmetry of the distribution of the Brownian motion that
%
%
\begin{eqnarray} \label{8*}
&&\int_{0}^{t}\int_{\mathbb{R}^{d}}E
\bigl[\chi_{\phi_{s}(W)}(x)\bigr]^{{1}/{2}%
}\,dx\,ds
\nonumber
\\
&&\qquad =\int_{0}^{t}\int_{\mathbb{R}^{d}}
\bigl(\mu\bigl(\phi_{s}^{-1}(x)\in W\bigr)
\bigr)^{{1}/{%
2}}\,dx\,ds
\nonumber
\\
&&\qquad \leq C\int_{0}^{t}\int_{\mathbb{R}^{d}}
\bigl(\mu(B_{s}+x\in W)\bigr)^{{1}/{4}%
}\,dx\,ds
\\
&&\qquad = C\int_{0}^{t}\int_{\mathbb{R}^{d}}\bigl(
\mu\bigl(B_{s}+x\in(-n,n)^{d}\bigr)\bigr)^{{1}/{4%
}}\,dx\,ds
\nonumber
\\
&&\qquad \leq C\int_{0}^{t}\biggl(2\int
_{0}^{\infty}\biggl(1-\Phi\biggl(\frac{-n+y}{\sqrt{s}}
\biggr)\biggr)^{%
{1}/{4}}\,dy\biggr)^{d}\,ds,\nonumber
\end{eqnarray}
where $\Phi$ is the standard normal distribution function.

On the other hand, we know that
\[
1-\Phi(x)\leq\frac{1}{2\pi x}\exp\bigl(-x^{2}/2\bigr)
\]
for all $x>0$ (see \cite{BorodinSalminen}).

So,
%
%
\begin{eqnarray} \label{9*}
&&\int_{0}^{t}\int_{\mathbb{R}^{d}}E
\bigl[\chi_{\phi_{s}(W)}(x)\bigr]^{{1}/{2}%
}\,dx\,ds
\nonumber
\\
&&\qquad \leq C\int_{0}^{t}\biggl(2\int
_{0}^{n}\biggl(1-\Phi\biggl(\frac{-n+y}{\sqrt{s}}
\biggr)\biggr)^{{1}/{%
4}}\,dy\nonumber
\\
&&\hspace*{60pt}{} +2\int_{n}^{\infty}
\biggl(1-\Phi\biggl(\frac{-n+y}{\sqrt{s}}\biggr)\biggr)^{{1}/{4}%
}\,dy\biggr)^{d}\,ds
\nonumber
\\
&&\qquad \leq K\int_{0}^{t}\biggl(\biggl(\int
_{0}^{n}\biggl(1-\Phi\biggl(\frac{-n+y}{\sqrt{s}}
\biggr)\biggr)^{{1}/{%
4}}\,dy\biggr)^{d}
\nonumber\\[-8pt]\\[-8pt]
&&\hspace*{63pt}{} +\biggl(\int
_{n}^{\infty}\biggl(1-\Phi\biggl(\frac{-n+y}{\sqrt{s}}
\biggr)\biggr)^{{1}/{4}%
}\,dy\biggr)^{d}\biggr)\,ds\nonumber
\\
&&\qquad \leq M\biggl(1+\int_{0}^{t}\biggl(\int
_{n}^{\infty}\biggl(\frac{\sqrt{s}}{2\pi
(y-n)}\exp
\bigl(-(y-n)^{2}/2s\bigr)\biggr)^{{1}/{4}}\,dy\biggr)^{d}\,ds
\biggr)\nonumber
\\
&&\qquad = M\biggl(1+\int_{0}^{t}\biggl(\int
_{0}^{\infty}\biggl(\frac{\sqrt{s}}{2\pi y}\exp
\bigl(-y^{2}/2s\bigr)\biggr)^{{1}/{4}}\,dy\biggr)^{d}
\,ds\biggr)\nonumber
\\
&&\qquad = M\biggl(1+\int_{0}^{t}\biggl(\int
_{0}^{\infty}\sqrt{s}\biggl(\frac{1}{2\pi y}\exp
\bigl(-y^{2}/2\bigr)\biggr)^{{1}/{4}}\,dy\biggr)^{d}
\,ds\biggr)
\leq L<\infty.\nonumber
\end{eqnarray}

Furthermore, since
\[
(Du)_{\varepsilon}\longrightarrow Du\qquad\mbox{in }L_{\mathrm{loc}}^{p}
\bigl(\mathbb{R}^{d}\bigr)
\]
for all $p>1$ and since
\[
\int_{\mathbb{R}^{d}}\bigl(\chi_{\phi_{s}(V)}(x)\theta\bigl(\phi
_{s}^{-1}(x)\bigr)b(s,x)\bigl\llvert\det\bigl(J
\phi_{s}^{-1}(x)\bigr)\bigr\rrvert\bigr)^{2}\,dx<
\infty\qquad\mbox{a.e.}
\]
because of the above estimates, we obtain
\begin{eqnarray*}
&&\int_{\mathbb{R}^{d}}\chi_{\phi_{s}(V)}(x)\theta\bigl(\phi
_{s}^{-1}(x)\bigr) \bigl((Du)_{\varepsilon}(s,x)\cdot
b(s,x)\bigr)\bigl\llvert\det\bigl(J\phi_{s}^{-1}(x)\bigr)
\bigr\rrvert \,dx
\\
&&\qquad \longrightarrow \int_{\mathbb{R}^{d}}\chi_{\phi_{s}(V)}(x)\theta
\bigl(\phi_{s}^{-1}(x)\bigr) \bigl((Du) (s,x)\cdot b(s,x)
\bigr)\bigl\llvert\det\bigl(J\phi_{s}^{-1}(x)\bigr)\bigr
\rrvert \,dx
\end{eqnarray*}
for $\varepsilon\searrow0 \mu\times ds$-a.e.

On the other hand, the latter expression w.r.t. $\varepsilon$ is
dominated by
the integrable term
\begin{eqnarray*}
&&\biggl(\int_{\mathbb{R}^{d}}\bigl(\chi_{\phi_{s}(V)}(x)\theta\bigl(
\phi_{s}^{-1}(x)\bigr)b(s,x)\bigl\llvert\det\bigl(J
\phi_{s}^{-1}(x)\bigr)\bigr\rrvert\bigr)^{2}\,dx
\biggr)^{1/2}
\\
&&\quad{}\times \biggl(\int_{\mathbb{R}^{d}}\chi_{\phi_{s}(W)}(x)
\bigl\llvert Du(s,x)\bigr\rrvert^{2}\,dx\biggr)^{1/2}.
\end{eqnarray*}
So, using dominated convergence it follows from (\ref{7*}) and (\ref
{9*}) that
%
%
\begin{eqnarray}\label{10*}
I_{1} &=&I_{1}(\varepsilon)\nonumber
\\
&\longrightarrow&E\biggl[\tau\int_{0}^{t}\int
_{\mathbb{R}^{d}}\chi_{\phi_{s}(V)}(x)\theta\bigl(
\phi_{s}^{-1}(x)\bigr) \bigl((Du) (s,x)\cdot b(s,x)\bigr)
\nonumber\\[-8pt]\\[-8pt]
&&\hspace*{129pt}{}\times
\bigl\llvert\det\bigl(J\phi_{s}^{-1}(x)\bigr)\bigr\rrvert
\,dx\,ds\biggr]\nonumber
\\
\eqntext{\mbox{as }\varepsilon \searrow 0.}
\end{eqnarray}

Similarly, we also get
%
%
\begin{eqnarray}\label{11*}
\qquad I_{2} &=&I_{2}(\varepsilon)\nonumber
\\
&\longrightarrow& -E\biggl[\tau\int_{0}^{t}\int
_{\mathbb{R}^{d}}\chi_{\phi_{s}(V)}(x)\theta\bigl(
\phi_{s}^{-1}(x)\bigr) (b\cdot Du) (s,x)
\\
&&\hspace*{105pt}{}\times \bigl\llvert\det
\bigl(J\phi_{s}^{-1}(x)\bigr)\bigr\rrvert \,dx\,ds\biggr]
\qquad \mbox{as }\varepsilon \searrow 0\nonumber
\end{eqnarray}
and
%
%
\begin{equation}
E\biggl[\tau\int_{V}\theta(x)u_{\varepsilon}\bigl(t,
\phi_{t}(x)\bigr)\,dx\biggr]\longrightarrow E\biggl[\tau\int
_{V}\theta(x)u\bigl(t,\phi_{t}(x)\bigr)\,dx\biggr]
\label{12*}
\end{equation}
as $\varepsilon\searrow0$.

In addition, because of the assumptions on $u_{0}$ it is clear that
\[
E\biggl[\tau\int_{V}\theta(x)u_{0,\varepsilon}(x)\,dx\biggr]
\longrightarrow E\biggl[\tau\int_{V}\theta(x)u_{0}(x)\,dx
\biggr]
\]
as $\varepsilon\searrow0$.

Altogether we can conclude that
\[
E\biggl[\tau\int_{\mathbb{R}^{d}}\theta(x)u\bigl(t,\phi_{t}(x)
\bigr)\,dx\biggr]=E\biggl[\tau\int_{%
\mathbb{R}^{d}}\theta(x)u_{0}(x)\,dx
\biggr]
\]
for all $\tau\in L^{\infty}(\Omega)$ and compactly supported smooth
functions $\theta$. Hence,
\[
u\bigl(t,\phi_{t}(x)\bigr)=u_{0}(x)
\]
$\mu\times dx$-a.e.

Since $\phi_{t}^{-1}(\cdot)$ satisfies the Lusin condition in Theorem
\ref%
{ChangeVariable} on bounded open subsets, we can find a $\Omega^{\ast}$
with $\mu(\Omega^{\ast})=1$ such that for all $\omega\in\Omega
^{\ast}$
\[
u(t,x)=u_{0}\bigl(\phi_{t}^{-1}(x)\bigr)\,dx
\mbox{-a.e.}
\]
Due to the continuity of $u$ with respect to time, the latter relation also
holds uniformly in $t$.

Finally, the Malliavin differentiability of (a version) of $u(t,x)$ is a
consequence of the fact that $\phi_{t}^{-1}(x)$ is Malliavin differentiable
(see \cite{PMNPZ}) and of the chain rule for Malliavin derivatives (see
\cite%
{Nualart}).
\end{pf*}

\begin{appendix}\label{sec4}
\section*{Appendix}
The following result which is due to \cite{DaPratoMalliavinNualart} provides
a compactness criterion for subsets of $L^{2}(\Omega;\mathbb{R}^{d})$ using
Malliavin calculus. See, for example, \cite{Nualart,Malliavin}
or \cite%
{DiNunnoOksendalProske} for more information about Malliavin calculus.

%
\begin{theorem}
\label{MCompactness}Let $ \{ ( \Omega,\mathcal{A},P );H \} $ be a Gaussian probability space, that is $ ( \Omega,%
\mathcal{A},P ) $ is a probability space and $H$ a separable closed
subspace of Gaussian random variables of $L^{2}(\Omega)$, which generate
the $\sigma$-field $\mathcal{A}$. Denote by $\mathbf{D}$ the derivative
operator acting on elementary smooth random variables in the sense that
\[
\mathbf{D}\bigl(f(h_{1},\ldots,h_{n})\bigr)=\sum
_{i=1}^{n}\partial_{i}f(h_{1},
\ldots,h_{n})h_{i},\qquad h_{i}\in H,f\in
C_{b}^{\infty}\bigl(%
\mathbb{R}^{n}
\bigr).
\]
Further, let $\mathbf{D}_{1,2}$ be the closure of the family of elementary
smooth random variables with respect to the norm
\[
\llVert F\rrVert_{1,2}:=\llVert F\rrVert_{L^{2}(\Omega
)}+\llVert
\mathbf{D}F\rrVert_{L^{2}(\Omega;H)}.
\]
Assume that $C$ is a self-adjoint compact operator on $H$ with dense image.
Then for any $c>0$, the set
\[
\mathcal{G}= \bigl\{ G\in\mathbf{D}_{1,2}\dvtx\llVert G\rrVert
_{L^{2}(\Omega)}+\bigl\llVert C^{-1} \mathbf{D} G\bigr\rrVert
_{L^{2}(\Omega;H)}\leq c \bigr\}
\]
is relatively compact in $L^{2}(\Omega)$.
\end{theorem}

In order to formulate compactness criteria useful for our purposes, we need
the following technical result which also can be found in \cite%
{DaPratoMalliavinNualart}.

%
\begin{lemma}
\label{DaPMN} Let $v_{s},s\geq0$ be the Haar basis of $L^{2}([0,1])$. For
any $0<\alpha<1/2$ define the operator $A_{\alpha}$ on $L^{2}([0,1])$ by
\[
A_{\alpha}v_{s}=2^{k\alpha}v_{s}\qquad
\mbox{if }s=2^{k}+j
\]
for $k\geq0,0\leq j\leq2^{k}$ and
\[
A_{\alpha}1=1.
\]
Then for all $\beta$ with $\alpha<\beta<(1/2)$, there exists a
constant $%
c_{1}$ such that
\[
\llVert A_{\alpha}f\rrVert\leq c_{1} \biggl\{ \llVert f\rrVert
_{L^{2}([0,1])}+ \biggl( \int_{0}^{1}\int
_{0}^{1}\frac{\llvert
f(t)-f(t^{\prime})\rrvert^{2}}{\llvert t-t^{\prime}\rrvert
^{1+2\beta}}\,dt \,dt^{\prime}
\biggr) ^{1/2} \biggr\}.
\]
\end{lemma}

A direct consequence of Theorem~\ref{MCompactness} and Lemma~\ref{DaPMN} is
now the following compactness criterion which is essential for the
proof of Theorem \ref{th13} and
Lemma~\ref{compactness}.

%
\begin{corollary}
\label{compactcrit} Let $X_{n}\in\mathbb{D}_{1,2}$, $n=1,2,\ldots,$ be a
sequence of \mbox{$\mathcal{F}_{1}$-}measu\-rable random variables such that there
are constants $\alpha>0$ and $C>0$ with
\begin{eqnarray*}
\sup_n E\bigl[ \| X_n \| ^2\bigr]
&\leq& C,
\\
\sup_{n}E \bigl[ \Vert D_{t}X_{n}-D_{t^{\prime}}X_{n}
\Vert^{2} \bigr] &\leq& C\bigl|t-t^{\prime}\bigr|^{\alpha}
\end{eqnarray*}
for $0\leq t^{\prime}\leq t\leq1$ and
\[
\sup_{n}\sup_{0\leq t\leq1}E \bigl[ \Vert
D_{t}X_{n}\Vert^{2} \bigr] \leq C,
\]
where $D_{t}$ denotes Malliavin differentiation. Then the sequence
$X_{n}$, $%
n=1,2,\ldots,$ is relatively compact in $L^{2}(\Omega;\mathbb{R}^{d} )$
($D_{t}$ stands for the Malliavin
derivative).
\end{corollary}
\end{appendix}




\printaddresses

\end{document}